\pdfoutput=1
\documentclass[11pt]{article}

\usepackage{preamble}

\title{Tree-like graphings, wallings, and median graphings of equivalence relations}
\author{Ruiyuan Chen, Antoine Poulin, Ran Tao, and Anush Tserunyan}

\date{}

\begin{document}

\maketitle

\begin{abstract}
We prove several results showing that every locally finite Borel graph whose large-scale geometry is ``tree-like'' induces a treeable equivalence relation.
In particular, our hypotheses hold if each component of the original graph either has bounded tree-width or is quasi-isometric to a tree, answering a question of Tucker-Drob.
In the latter case, we moreover show that there exists a Borel quasi-isometry to a Borel forest, under the additional assumption of (componentwise) bounded degree.
We also extend these results on quasi-treeings to Borel proper metric spaces.
In fact, our most general result shows treeability of countable Borel equivalence relations equipped with an abstract wallspace structure on each class obeying some local finiteness conditions, which we call a \emph{proper walling}.
The proof is based on the Stone duality between proper wallings and median graphs, i.e., CAT(0) cube complexes.
Finally, we strengthen the conclusion of treeability in these results to hyperfiniteness in the case where the original graph has one (selected) end per component, generalizing the same result for trees due to Dougherty--Jackson--Kechris.
\let\thefootnote=\relax
\end{abstract}

\tableofcontents

\section{Introduction}
\label{sec:intro}

This paper is a contribution to the descriptive combinatorics and large-scale geometry of Borel equivalence relations and graphs.
A \defn{countable Borel equivalence relation (CBER)} $E$ on a standard Borel space $X$ is a Borel equivalence relation $E \subseteq X^2$ each of whose classes is countable.
CBERs have been widely studied over the past few decades in connection with many different areas, including topological and measurable dynamics, probability, operator algebras, logic and classification problems, and combinatorics.
For recent surveys, see \cite{Kcber}, \cite{Piklfg}, \cite{KMdgc}.

A recurring theme in the theory of CBERs is their stratification according to the types of graphs and other combinatorial structures that may be uniformly assigned to each equivalence class; see \cite{CKstr}, \cite{JKLcber} for the general aspects of this so-called \emph{structurability} theory.
In this paper, we are primarily concerned with graph-theoretic structurability.
A \defn{Borel graphing} of a CBER $E \subseteq X^2$ is a (locally countable) Borel graph $G \subseteq X^2$ which generates $E$ as an equivalence relation, i.e., whose connectedness relation is precisely $E$.
One regards such a $G$ as the ``Borel assignment'' to each equivalence class $C \in X/E$ of the connected component $G|C$.
By requiring each $G|C$ to be some specific type of connected graph, one then obtains various natural restricted subclasses of CBERs.
In particular, the class of \defn{treeable} CBERs, for which it is possible to find a graphing all of whose components are trees (a \defn{treeing}), plays a similar role in the theory of CBERs as free groups in group theory, and has been extensively studied from this perspective; see e.g., \cite{Adatree}, \cite{Gabcost}, \cite{JKLcber}, \cite{GLf2}, \cite{Hjotree}, \cite{Miltree}, \cite{CGMTtree}.

The main results in this paper provide new sufficient criteria for treeability of CBERs.
We work purely in the Borel context; however, our results are new even in the measurable setting (where one is allowed to discard a null set).

\subsection{Notions of tree-like graphs}
\label{sec:intro-treelike}

Two metric spaces $X, Y$ are \defn{quasi-isometric} if they are isometric up to a bounded multiplicative and additive error; see \cref{def:coarse}.
A metric space quasi-isometric to a simplicial tree is called a \defn{quasi-tree}.
Quasi-isometry of metric spaces and especially of Cayley graphs of groups has played a central role in metric geometry and geometric group theory since the work of Gromov \cite{Groggt2}.

Our work was initially motivated by a classical result in this tradition, which states that a finitely generated group which is a quasi-tree must be virtually free, i.e., contain a finite-index free subgroup; see \cite[7.19]{GHhyp}.
In the context of CBERs, a treeable CBER is the analogue of a free group, since the orbit equivalence relation of a free action of a free group is treeable;%
\footnote{In fact a converse is true as well, up to ``measuring the number of generators'' \cite{Hjocost}.}
moreover, free actions of virtually free groups are also known to be treeable \cite[3.4]{JKLcber}.
Thus, a natural question is whether the class of treeable CBERs is robust in some sense under quasi-isometry, and this was asked by Robin Tucker-Drob at the 2015 Annual North American Meeting of the ASL. We give a positive answer to this question:

\begin{theorem}[\cref{thm:cber-quasitreeable-treeable}]
\label{intro:thm:cber-quasitreeable-treeable}
If a CBER admits a locally finite Borel graphing whose components are quasi-trees, then it is Borel treeable.
\end{theorem}

We note that the local finiteness assumption is natural, since the conclusion of treeability is equivalent to locally finite treeability \cite[3.12]{JKLcber}, and since without this assumption every CBER would admit the complete graphing as a quasi-treeing (each of whose components is quasi-isometric to a point).
By strengthening the assumption, we may refine the conclusion from the level of the equivalence relation to the graph itself:

\begin{theorem}[\cref{thm:cber-quasitreeable-bddeg}]
\label{intro:thm:cber-quasitreeable-bddeg}
If a Borel graph has every component a bounded degree quasi-tree, then it is Borel quasi-isometric to a componentwise bounded degree Borel forest.
\end{theorem}

\begin{remark}
Both preceding results admit (easy) natural extensions to Borel (pseudo)metrics which are quasi-trees on each equivalence class; see \cref{thm:cber-quasitreeable-metric}.
\end{remark}

\medskip
A different notion of ``tree-like graph'' originates from finite combinatorics, namely the graph minor theory of Robertson--Seymour \cite{RSminor10}.
A \defn{tree decomposition} of a graph $G$ over a tree $T$ is a homomorphism from $G$ to the intersection graph on the subtrees of $T$, and the graph $G$ has \defn{tree-width $<N$} if it has a tree decomposition such that each vertex of $T$ is in the images of at most $N$ vertices of $G$; see \cref{def:treedec}.
Tree decompositions have been considered by \cite{Cartw} for infinite graphs, where \emph{bounded tree-width} becomes a natural notion of ``tree mod finite''.
We prove

\begin{theorem}[\cref{thm:cber-bddtw-treeable}]
\label{intro:thm:cber-bddtw-treeable}
If a CBER admits a locally finite Borel graphing with components of bounded tree-width, then it is Borel treeable.
\end{theorem}

Note that a bounded degree quasi-tree graph has bounded tree-width, but a general locally finite quasi-tree need not, nor conversely does bounded tree-width imply quasi-tree; see \cref{ex:dense-nonquasitree-nonbddtw}.
Thus, \cref{intro:thm:cber-quasitreeable-treeable,intro:thm:cber-bddtw-treeable} are not directly comparable.

We were informed in a private communication that a version of \cref{intro:thm:cber-bddtw-treeable} has also been proved by Héctor Jardón-Sánchez \cite{Hector}, using different methods.

\subsection{Dense families of cuts}

Both of our concrete \cref{intro:thm:cber-quasitreeable-treeable,intro:thm:cber-bddtw-treeable} are instances of the following abstract result, which aims to capture the general concept of a ``large-scale tree-like'' graph.

In both concrete results, the ``tree-likeness'' of the graph $G$ we start with is witnessed on each component $C$ by a function/binary relation $G|C -> T_C$ from that component to a genuine tree $T_C$ which ``roughly preserves'' the structure in some sense, namely a quasi-isometry/tree decomposition.
The issue is to find such $T_C$ in a canonical Borel manner simultaneously for each component $C$.
Rather than proceed directly, we use that the mere \emph{existence} of such a $G|C -> T_C$ allows us to transport the edges of the $T_C$ over to a \emph{canonical} set of ``pseudo-edges'' of $G|C$ witnessing that it is ``tree-like''.
These ``pseudo-edges'' take the form of \emph{cuts} in the graph $G|C$; and the fact that they render the graph ``tree-like'' is captured by conditions \cref{intro:thm:cber-cuts-dense-treeable:finsep,intro:thm:cber-cuts-dense-treeable:dense} in the following theorem, which is our main result concerning treeability of Borel graphs.

\begin{theorem}[\cref{thm:cber-cuts-dense-treeable}]
\label{intro:thm:cber-cuts-dense-treeable}
Let $E \subseteq X^2$ be a CBER with a locally finite Borel graphing $G$.
Suppose there exists a Borel assignment $X/E \ni C |-> \@H(C) \subseteq 2^C$ to each component $C$ of a collection of sets of vertices $H \subseteq C$ with finite boundary (the ``cuts'') such that both $H$ and $C \setminus H$ are connected, so that the following conditions hold for each $C$:
\begin{enumerate}[label=(\roman*)]
\item \label{intro:thm:cber-cuts-dense-treeable:finsep}
each vertex $x \in C$ lies on the boundary of only finitely many $H \in \@H(C)$;
\item \label{intro:thm:cber-cuts-dense-treeable:dense}
each end of the graph $G|C$ has a neighborhood basis in $\@H(C)$.
\end{enumerate}
Then $E$ is Borel treeable.
\end{theorem}

Here in \cref{intro:thm:cber-cuts-dense-treeable:dense} we mean neighborhoods in the \emph{end compactification} $\^C^G$ of $G|C$, i.e., the compact zero-dimensional space of ends together with vertices as isolated points converging to ends.
We express this condition by saying that $\@H(C)$ is \defn{dense towards ends} of $G|C$.
The notion of \emph{Borel assignment} $C |-> \@H(C)$ can be made precise by e.g., identifying each $H \in \@H(C)$ with its boundary in the standard Borel space of finite subsets, or by considering an $E$-bundle (see \cref{def:walling}).

From the above abstract result, \cref{intro:thm:cber-quasitreeable-treeable} follows by taking $\@H(C)$ to be the collection $\@H_{\diamle R}(C)$ of cuts with boundary of diameter $\le R$ for some $R < \infty$ (see \cref{thm:quasitree-dense}).
\Cref{intro:thm:cber-bddtw-treeable} follows by taking $\@H(C)$ to be the collection $\@H_{\cardle N}(C)$ of cuts with minimum cardinality of the inner and outer boundaries $\le N$ for some $N < \infty$ (see \cref{thm:bddtw-dense}).

\begin{example}
\label{ex:dense-nonquasitree-nonbddtw}
Consider graphs of the following form:
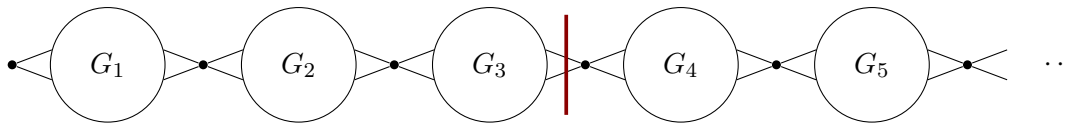
\begin{center}
\begin{tikzpicture}[x=1in,y=0.5cm]
\node(n0) [bullet] {};
\foreach \i in {1,...,5} {
    \pgfmathsetmacro\j{int(\i-1)}
    \node(n\i) [bullet] at (\i,0) {};
    \node(g\i) [draw, circle, minimum width=.6in] at ({\i-0.5},0) {$G_\i$};
    \draw[graph edge]
        (n\j) edge (g\i.165) edge (g\i.195)
        (n\i) edge (g\i.15) edge (g\i.-15);
};
\node(g6) [circle, minimum width=.6in] at (5.5,0) {$\dotsb$};
\draw[graph edge]
    (n5) edge (g6.165) edge (g6.195);
\draw[graph cut] (2.9,1) -- (2.9,-1);
\end{tikzpicture}
\end{center}
where the $G_n$ are finite graphs which are increasingly ``complex'' in the various senses below.
Note that regardless of what the $G_n$ are, both $\@H_{\diamle 2}$ and $\@H_{\cardle 1}$ are dense towards the single end (the example may clearly be generalized to produce multi-ended graphs); the bolded line depicts a typical cut in both of these collections.
\begin{itemize}
\item
If the $G_n$ are increasingly large finite cliques, we get a quasi-tree with unbounded tree-width.
\item
If the $G_n$ are increasingly long finite cycles, we get a non-quasi-tree with bounded tree-width.
\item
If the $G_n$ contain both increasingly large cliques and increasingly long cycles (without shortcuts), we get a graph which neither is a quasi-tree nor has bounded tree-width.
\end{itemize}
This shows that \cref{intro:thm:cber-quasitreeable-treeable,intro:thm:cber-bddtw-treeable} are incomparable with each other, and that \cref{intro:thm:cber-cuts-dense-treeable} strictly generalizes both of them combined.
\end{example}

\subsection{Wallspaces and median graphs}
\label{sec:intro-walls}

In order to prove \cref{intro:thm:cber-cuts-dense-treeable}, one needs to canonically convert a ``nice'' family of cuts in a graph into a tree.
Broadly speaking, our method belongs to the tradition of Stallings' theorem on multi-ended groups \cite{Sta3d}, the techniques of which have been widely influential (see \cite{Drutu-Kapovich:GGT} for a survey), including in Borel and measurable combinatorics \cite{Ghyleaf}, \cite{GLf2}, \cite{Tsestallings}.

However, there are difficulties in directly applying the Stallings machinery to our setting, which are ultimately rooted in the precise meaning of ``canonical'' in Borel combinatorics (see \cref{rmk:canonical} below for more on this point).
In the usual Stallings proof, one reduces the initial collection of cuts to an automorphism-invariant subfamily of pairwise nested cuts, from which one builds a tree with those edges by taking the vertices to be ``ultrafilters''%
\footnote{In this paper, we prefer the term ``orientation''; see \cref{def:orientation,ftn:orientation}.}
of those cuts.
But starting from a Borel graph, such a tree will usually not be standard Borel if the nested cuts are too sparse.
On the other hand, the requirement of full automorphism-invariance is stronger than needed (again see \cref{rmk:canonical}).
This mismatch is the reason that the Borel version of Stallings' theorem in \cite{Tsestallings} does not apply in our context.

Our approach is therefore to work with the \emph{entire} (non-nested) collection of cuts $\@H(C)$ as in \cref{intro:thm:cber-cuts-dense-treeable}.
Again, we in fact prove a more general result that isolates the key features of such a collection.
By a \defn{proper wallspace}, we mean a set $X$ equipped with a family $\@H(X) \subseteq 2^X$ of subsets obeying certain ``local finiteness'' conditions; see \cref{def:walls-proper} for the precise definition.
(The term ``wallspace'' derives from metric geometry \cite{Nwalls}, \cite{CNwalls}.)
We call a CBER $E \subseteq X^2$ \defn{properly wallable} if there is a Borel assignment of such collections $X/E \ni C |-> \@H(C) \subseteq 2^C$ (called a \defn{proper walling} of $E$); see \cref{def:walling}.
We prove the following general treeability result, no longer just about graphs:

\begin{theorem}[\cref{thm:cber-wallable-treeable}]
\label{intro:thm:cber-wallable-treeable}
A CBER is properly wallable iff it is treeable.
\end{theorem}

Given this, the proof of \cref{intro:thm:cber-cuts-dense-treeable} reduces to showing that conditions \cref{intro:thm:cber-cuts-dense-treeable:finsep,intro:thm:cber-cuts-dense-treeable:dense} from there imply that $C |-> \@H(C)$ is a proper walling.
This is done in \cref{thm:graph-walling-fin,thm:cut-proper}.

Finally, we discuss the proof of \cref{intro:thm:cber-wallable-treeable}.
Since we do not assume $\@H(C)$ to be nested, the ``ultrafilters'' on it form not a tree, but a generalization thereof.
A \defn{median graph} is, roughly speaking, a graph with the same well-behaved notions of ``flatness'' and ``convexity'' as in a tree, but not necessarily required to be ``one-dimensional''; see \cref{def:median} for the precise definition.
Median graphs have been well-studied in geometry as (1-skeleta of) \defn{CAT(0) cube complexes}, in which form they were previously used in a Borel combinatorics context by \cite{HSShyp};
and they are also well-known in lattice theory as interval-finite \defn{median algebras}.
See \crefrange{sec:median}{sec:duality} for detailed background and references on median graphs.

Most importantly for our purposes, there is a Stone-type duality for median graphs, due to Isbell \cite{Isbmed} and Werner \cite{Werdual}, that we review in \cref{thm:duality}.
This duality shows that a general walling $\@H(C)$ as in \cref{intro:thm:cber-wallable-treeable} provides exactly the data required to construct (not a tree but) a median graph, with the sets in $\@H(C)$ forming (not the edges but) the ``hyperplanes''.
Moreover, the additional requirement that $\@H(C)$ be proper implies that each of these hyperplanes is finite (see \cref{thm:med-hyp-locfin,def:walls-proper}).
\Cref{intro:thm:cber-wallable-treeable} thus reduces to

\begin{theorem}[\cref{thm:cber-median}]
\label{intro:thm:cber-median}
Every Borel median graphing of a CBER with finite hyperplanes has a Borel subtreeing.
\end{theorem}

This result can be regarded as the combinatorial heart of our paper, to which all of the aforementioned treeability results reduce via abstract machinery, as summarized in \cref{fig:flowchart}.

\begin{figure}[htb]
\centering
\begin{tikzpicture}[
    every node/.append style={draw, rectangle, align=center},
    every edge/.append style={->, every node/.style={draw=none, coordinate}},
    every label/.append style={rectangle, font=\scriptsize},
]
\node(tree) {tree};
\node(median)
    [label={below:\ref{intro:thm:cber-median}~(\ref{thm:cber-median})}]
    at ($(tree)+(-2.5,0)$) {median\\graph\\w/ fin.\\hyperpl.}
    edge[""{label=below:\ref{thm:med-subtree}}]
    (tree);
\node(walls)
    [label={below:\ref{intro:thm:cber-wallable-treeable}~(\ref{thm:cber-wallable-treeable})}]
    at ($(median)+(-3.5,0)$) {proper\\wallspace}
    edge[""{label={below:\ref{thm:duality}+\ref{thm:med-hyp-locfin}}}]
    (median);
\node(cuts)
    [label={below:\ref{intro:thm:cber-cuts-dense-treeable}~(\ref{thm:cber-cuts-dense-treeable})}]
    at ($(walls)+(-3.2,0)$) {graph\\with\\dense\\cuts}
    edge[""{label={below:\ref{thm:graph-walling-fin}+\ref{thm:cut-proper}}}]
    (walls);
\node(qtree)
    [label={below:\ref{intro:thm:cber-quasitreeable-treeable}~(\ref{thm:cber-quasitreeable-treeable})}]
    at ($(cuts)+(-3,1)$) {quasi-tree}
    edge[""{label={above:\ref{thm:quasitree-dense}}}]
    (cuts);
\node(bddtw)
    [label={below:\ref{intro:thm:cber-bddtw-treeable}~(\ref{thm:cber-bddtw-treeable})}]
    at ($(cuts)+(-3,-1)$) {bounded\\tree-width}
    edge[""{label={below:\ref{thm:bddtw-dense}}}]
    (cuts);
\end{tikzpicture}
\caption{The process of converting ``tree-like'' structures into trees.
Number on arrow refers to result(s) showing that step for a countable structure.
Number on box refers to end result on treeability of CBERs.}
\label{fig:flowchart}
\end{figure}
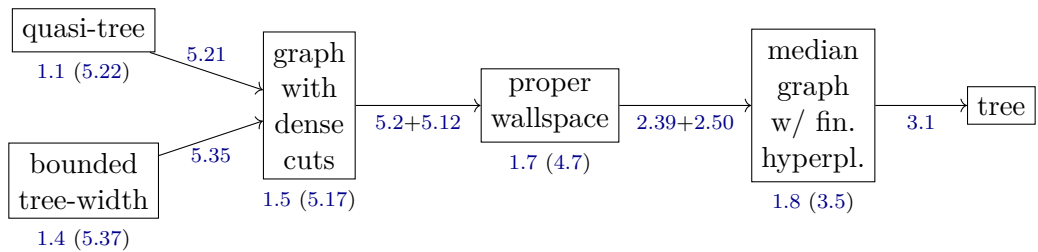

\subsection{The case of one end}

It is well-known that a CBER with a one-ended treeing, or more generally a treeing with a Borel selection of one end per component, is hyperfinite \cite{DJKhyp}.
This suggests that the same might hold for the various ``tree-like'' graphings considered above.
In \cref{sec:end-hypf} we show that this is the case:

\begin{theorem}[\cref{thm:cber-quasitreeing-bddtw-end-hypf}]
\label{intro:thm:cber-quasitreeing-bddtw-end-hypf}
If a CBER admits a locally finite Borel graphing whose components are quasi-trees or of bounded tree-width, together with a Borel selection of one end per component, then it is hyperfinite.
\end{theorem}

As before, this is in fact an instance of a more general result on Borel graphs equipped with dense families of cuts in each component (see \cref{thm:cber-cuts-end-hypf}), which in turn reduces to a core result for median graphs:

\begin{theorem}[\cref{thm:cber-median-end-hypf}]
\label{intro:thm:cber-median-end-hypf}
If a CBER admits a median graphing with finite hyperplanes and a Borel selection of one end per component, then it is hyperfinite.
\end{theorem}

The proof of this result is based on defining a single transformation that ``moves towards'' the chosen end in each component (see \cref{fig:end-corner}), generalizing the same idea for trees, and further showcasing the usefulness of the rich geometry of median graphs in Borel combinatorics.

On the other hand, in the restricted case of ``tree-like'' graphings which have only one end in total, we give in \cref{sec:oneended} a much simpler proof of hyperfiniteness, that is conceptually based on a similar idea of ``moving towards'' the one end, but can be read independently of the rest of the paper, without knowing the definitions of ``median graph'', ``walling'', or even ``end''.
In fact, there we (re)formulate a simple notion of \defn{one-ended proper walling}, which easily implies hyperfiniteness, and which is easily implied by one-ended quasi-treeing or bounded tree-width graphing, all comfortably fitting onto two pages; see \ref{def:walls-oneended}--\ref{thm:cber-quasitreeing-bddtw-oneended}.

\subsection{Remark on Borel versus countable combinatorics}
\label{rmk:canonical}

As usual in countable Borel combinatorics, and as alluded to in \cref{sec:intro-walls}, our results are in some sense not about ``Borel'' structures at all, but rather about ``canonical'' constructions of countable structures.
That is, in order to prove a theorem of the form
\begin{enumerate}[label=(\alph*)]
\item \label{rmk:canonical:borel}
every CBER admitting a Borel ``$\heartsuit$ing'' (e.g., quasi-treeing) also admits a ``$\clubsuit$ing'' (e.g., treeing)
\end{enumerate}
one really proves
\begin{enumerate}[resume*]
\item \label{rmk:canonical:ctbl}
every countable ``$\heartsuit$'' can be canonically turned into a ``$\clubsuit$''
\end{enumerate}
where ``canonical'' refers to (countable) operations which can be uniformly performed on each equivalence class of a CBER.
For example, picking a single point is not allowed, nor is quotienting by an arbitrary equivalence relation (even if that relation is fully ``canonical'', e.g., automorphism-invariant).
On the other hand, certain non-automorphism-invariant operations \emph{are} allowed, such as taking a countable coloring of the intersection graph on all finite subsets \cite[7.3]{KMtopics}.

In this paper, we fully embrace this perspective, by largely working explicitly with countable structures throughout, and only occasionally pointing out why the preceding countable constructions are ``canonical'' in the above sense.
For instance, \cref{intro:thm:cber-median} is really the following result:

\begin{theorem}[\cref{thm:med-subtree}]
\label{intro:thm:med-subtree}
Let $(X,G)$ be a countable median graph with finite hyperplanes.
Then we may construct a canonical subtree $T \subseteq G$.
\end{theorem}

\begin{remark}
In fact the equivalence between \cref{rmk:canonical:borel} and \cref{rmk:canonical:ctbl} above can be made fully precise, by taking \cref{rmk:canonical:ctbl} to mean that there is a certain kind of interpretation in the infinitary logic $\@L_{\omega_1\omega}$ from the theory of $\clubsuit$ to the theory of $\heartsuit$ expanded with two additional pieces of structure, namely, the aforementioned countable coloring of the intersection graph on finite subsets and a countable family of subsets separating points.
By interpretation, we mean that in any countable $\heartsuit$-structure $X$ we can define a $\clubsuit$-structure in $X$ (by a collection of $\@L_{\omega_1\omega}$ formulas) and this definition does not depend on $X$.
We do not use this precise equivalence at all in this paper, but refer the interested reader to \cite[arXiv version, Appendix~B]{CKstr} and \cite{BCstr} for details.
\end{remark}

\subsection*{Acknowledgments}

We would like to thank Robin Tucker-Drob for asking and bringing to our attention the question of treeability of Borel quasi-trees, as well as
Clinton Conley, Oleg Pikhurko, Robin Tucker-Drob, and Felix Weilacher for helpful discussions and suggestions concerning our initial results and potential avenues for generalization.
We also thank Alexander Kechris and Forte Shinko for useful questions and suggestions.
Finally, we thank the anonymous referee for several helpful comments which improved our exposition.
R.C.\ was supported by NSF grant DMS-2224709.
A.P.\ was supported by NSERC CGS D 569502-2022.
R.T.\ was supported by NSERC CGS M and FRQNT Bourse de maîtrise en recherche.
A.Ts.\ was supported by NSERC Discovery Grant RGPIN-2020-07120.

\section{Preliminaries}
\label{sec:prelim}

\subsection{Graphs}
\label{sec:graph}

A \defn{graph} $G$ on vertex set $X$ will mean a symmetric irreflexive binary relation $G \subseteq X^2$, where $(x,y) \in G$ represents the oriented edge from $x$ to $y$.
We will refer to the graph by $G$ or $(X,G)$.

Given a connected graph $(X,G)$, we let $d = d_G : X^2 -> [0,\infty)$ denote the path metric, and $\Ball_r(x)$ denote the closed ball of radius $r$ around $x \in X$.
More generally, for $A \subseteq X$,
\begin{align*}
\Ball_r(A) := \bigcup_{x \in A} \Ball_r(x).
\end{align*}

\begin{definition}
\label{def:boundary}
For a subset $A \subseteq X$ of vertices, we write $\neg A := X \setminus A$.
\leavevmode
\begin{itemize}
\item  The \defn{inner vertex boundary} of $A$ is $\partial_\iv A := A \cap \Ball_1(\neg A)$.
\item  The \defn{outer vertex boundary} of $A$ is $\partial_\ov A := \Ball_1(A) \cap \neg A = \partial_\iv \neg A$.
\item  The \defn{total vertex boundary} of $A$ is $\partial_\v A := \partial_\iv \cup \partial_\ov A$.
\item  The \defn{inward edge boundary} of $A$ is $\partial_\ie A := G \cap (\neg A \times A) = G \cap (\partial_\ov A \times \partial_\iv A)$.
\item  The \defn{outward edge boundary} of $A$ is $\partial_\oe A := G \cap (A \times \neg A) = G \cap (\partial_\iv A \times \partial_\ov A) = \partial_\ie \neg A$.
\end{itemize}
Note that even though our graphs are symmetric, we always consider \emph{oriented} edge boundaries.
When the graph $G$ is not clear from context, we will specify it via superscript, e.g., $\partial^G_\iv A = \partial_\iv A$.
\end{definition}

\begin{definition}
\label{def:int}
The \defn{interval} $[x,y]$ between $x, y \in X$ is the union of all geodesics between $x, y$:
\begin{equation*}
z \in [x,y]  \coloniff  d(x,z) + d(z,y) = d(x,y).
\end{equation*}
For $z \in [x,y]$, we say $z$ is \defn{between} $x,y$, also denoted as the ternary relation
\begin{equation*}
\betw{x--z--y}  \coloniff  z \in [x,y].
\end{equation*}
\end{definition}

The following formal properties of intervals/betweenness are easily seen:
\begin{gather}
\label{eq:int-refl-sym}
x, y \in [x,y] = [y,x], \\
\label{eq:int-antisym}
(\betw{x--y--z} \AND \betw{y--x--z})  \iff  x = y, \\
\label{eq:int-assoc}
(\betw{w--x--y} \AND \betw{w--y--z}) \iff (\betw{w--x--z} \AND \betw{x--y--z}).
\end{gather}
Both sides of this last relation hold iff there exists a geodesic from $w$ to $x$ to $y$ to $z$, in which case we also write $\betw{w--x--y--z}$.
The notation $\betw{x_0--x_1--x_2--\dotsb--x_n}$ is defined similarly.

\begin{definition}
A set of vertices $A \subseteq X$ is \defn{convex} if $x, y \in A \implies [x,y] \subseteq A$, i.e., every geodesic between every two vertices in $A$ is contained in $A$.
This clearly implies $A$ is connected or empty.

For $A \subseteq X$, let $\cvx(A)$ denote the \defn{convex hull} of $A$, i.e., the smallest convex set containing $A$.
\end{definition}

For detailed information on the axiomatics of intervals and convexity in graphs and related structures, see \cite{vdVcvx}.

\begin{convention}
\label{conv:H}
We will be interested in several families $\@H \subseteq 2^X$ of subsets of the vertex set $X$ of a graph (and later, more generally an arbitrary set $X$), which have the property of being closed under the complement operation $\neg : 2^X -> 2^X$.
For example, $\@H_\cvx$ will denote the family of convex sets whose complement is also convex; see \cref{def:half}.

For such families $\@H$, we will always use the plain symbol $\@H$ to denote the family including the two \defn{trivial} subsets $\emptyset, X \in \@H$, while $\@H^* := \@H \setminus \{\emptyset, X\}$ will denote the \defn{nontrivial} elements of $\@H$.
We also write $\@H^G(X) = \@H(X) = \@H$ (e.g., $\@H_\cvx^G(X)$) when the vertex/edge set needs to be specified.
\end{convention}

\subsection{Median graphs}
\label{sec:median}

\begin{definition}
\label{def:median}
A \defn{median graph} $(X,G)$ is a connected graph such that for any $x, y, z \in X$, the intersection of the pairwise intervals between them is a singleton; this single vertex is called the \defn{median} of $x, y, z$, denoted $\ang{x,y,z}$:
\begin{equation*}
[x,y] \cap [y,z] \cap [z,x] = \{\ang{x,y,z}\}.
\end{equation*}
\end{definition}

Median graphs have been widely studied from many different perspectives and under different guises, including in combinatorics, geometry and group theory (as \emph{CAT(0) cube complexes}), universal algebra and lattice theory (as interval-finite \emph{median algebras}), and logic and computer science (in relation to the \emph{2-satisfiability} problem).
For comprehensive surveys, see \cite{BHmed}, \cite{Rolmed}, \cite{Bowmed}.
In this and the next few subsections, we give a brief, self-contained exposition of the basic theory of median graphs needed in this paper from a purely combinatorial perspective; we hope that such an account will also facilitate future applications of median graphs in Borel combinatorics.

In the rest of this subsection, let $(X,G)$ be a median graph.

\begin{remark}
For $x, y \in X$, $\ang{x,y,\blank} : X -> X$ is idempotent, with image = fixed points = $[x,y]$.
Thus the median operation and the interval relation may be (positively) defined from each other.
\end{remark}

\begin{definition}
\label{def:homom}
A \defn{median homomorphism} $f : (X,G) -> (Y,H)$ between two median graphs is a function between the vertex sets $f : X -> Y$ which preserves the ternary median operation $\ang{\blank,\blank,\blank}$, or equivalently by the preceding remark, the ternary interval relation.

(This neither implies nor is implied by being a graph homomorphism.
See however \cref{rmk:homom-cvx}.)
\end{definition}

\begin{lemma}
\label{thm:med-meet}
For any $x, y, z \in X$, we have $[x,y] \cap [x,z] = [x,\ang{x,y,z}]$; and for any $w$ in this set, we have $\ang{w,y,z} = \ang{x,y,z}$.
\end{lemma}
\begin{proof}
$[x,y] \cap [x,z] \supseteq [x,\ang{x,y,z}]$ follows from \cref{eq:int-assoc}.

Now let $w \in [x,y] \cap [x,z]$.
Then $\ang{w,y,z} \in [w,y] \cap [w,z] \cap [y,z] \subseteq [x,y] \cap [x,z] \cap [y,z]$ again by \cref{eq:int-assoc}, whence $\ang{w,y,z} = \ang{x,y,z}$.

Since $w \in [x,y]$ and $\ang{x,y,z} = \ang{w,y,z} \in [w,y]$, again by \cref{eq:int-assoc}, $w \in [x,\ang{x,y,z}]$.
\end{proof}

\begin{definition}
\label{def:cone}
For two vertices $x, y \in X$, the \defn{cone} at $y$ away from $x$ is
\begin{equation*}
\cone_x y := \{z \in X \mid \betw{x--y--z}\}.
\end{equation*}
\end{definition}

\begin{lemma}
\label{thm:cvx-up}
For any $x, y \in X$, the set $\cone_x y$ is convex.
\end{lemma}
\begin{proof}
Let $a, b \in \cone_x y$ and $c \in [a,b]$; we will show
\begin{equation*}
\betw{a--\ang{a,c,x}--\ang{a,b,x}--y--x},
\end{equation*}
which by \cref{eq:int-assoc} along with $\betw{c--\ang{a,c,x}--x}$ gives $\betw{c--y--x}$ as desired.
By \cref{eq:int-assoc}, it suffices to show:
\begin{itemize}
\item  $\betw{a--\ang{a,b,x}--x}$ by definition of $\ang{a,b,x}$.
\item  $\betw{\ang{a,b,x}--y--x}$, i.e., $y \in [x, \ang{a,b,x}]$, by \cref{thm:med-meet}.
\item  $\betw{a--\ang{a,c,x}--\ang{a,b,x}}$ again by \cref{thm:med-meet}, because $\ang{a,c,x} \in [a,x]$ and $\ang{a,c,x} \in [a,b]$, the latter because $\ang{a,c,x} \in [a,c] \subseteq [a,b]$ by \cref{eq:int-assoc}.
\qedhere
\end{itemize}
\end{proof}

\begin{proposition}
\label{thm:proj}
For any $\emptyset \ne A \subseteq X$ and $x \in X$, there is a unique point in the convex hull of $A$ which is between $x$ and every point in $A$, called the \defn{projection} $\proj_A(x)$ of $x$ toward $A$.%
\footnote{The projection is also known as the \emph{gate} \cite{vdVcvx}, \cite{Bowmed} (especially in connection with \cref{thm:proj-proj}); however, that term also has other meanings \cite{Rolmed}.}
Thus
\begin{align*}
\cvx(A) \cap \bigcap_{a \in A} [x,a] = \{\proj_A(x)\}.
\end{align*}
Moreover, $\proj_A = \proj_{\cvx(A)}$, i.e., $\proj_A(x)$ is also between $x$ and every point in $\cvx(A)$.
\end{proposition}
\begin{proof}
Pick any $a_0 \in A$, and inductively given $a_n$, if there is $a \in A$ such that $a_n \not\in [x,a]$, then put $a_{n+1} := \ang{x,a,a_n}$.
Then $\betw{a_0--a_1--a_2--\dotsb--x}$ by repeated application of \cref{eq:int-assoc}, so this sequence must terminate in at most $d(a_0,x)$ steps at a point in $\cvx(A)$ which is between $x$ and every point in $A$; this proves existence.

$\proj_A(x)$ is also between $x$ and every point in $\cvx(A)$, since $\cone_x{\proj_A(x)}$ is convex and contains $A$.
Now for uniqueness: if there were two such points $a, b \in \cvx(A)$, then we would have $\betw{x--a--b}$ and $\betw{x--b--a}$, whence $a = b$ by \cref{eq:int-antisym}.
\end{proof}

\begin{example}
$\ang{a,b,x} = \proj_{\{a,b\}}(x)$.
It follows that $[a,b] = \im(\ang{a,b,\blank}) = \cvx(\{a,b\})$ is convex.
\end{example}

In light of this example, the following generalizes \cref{thm:med-meet}, and is proved similarly:

\begin{lemma}
\label{thm:med-proj}
For any $x \in X$ and $\emptyset \ne A \subseteq X$, we have $\bigcap_{a \in A} [x,a] = [x,\proj_A(x)]$; and for any $w$ in this set, we have $\proj_A(w) = \proj_A(x)$.
\qed
\end{lemma}

\begin{remark}
\label{thm:homom-proj}
It follows from the proof of \cref{thm:proj} that for a median homomorphism $f : (X,G) -> (Y,H)$ between median graphs, for any $A \subseteq X$ and $x \in X$, we have
\begin{align*}
f(\proj_A(x)) = \proj_{f(A)}(f(x)).
\end{align*}
Indeed, $\proj_A(x) = \dotsb\ang{x,a_3,\ang{x,a_2,\ang{x,a_1,a_0}}}$; this expression is preserved by $f$.
\end{remark}

\begin{proposition}
\label{thm:proj-homom}
$\proj_A : X ->> \cvx(A)$ is the unique median homomorphism fixing $\cvx(A)$.
\end{proposition}
\begin{proof}
First we check $\proj_A$ is a median homomorphism.
Let $x, y, z \in X$ with $\betw{x--y--z}$.
Then
\begin{align*}
w := \ang{\proj_A(y),\proj_A(x),\proj_A(z)}
\in [\proj_A(y),x] \cap [\proj_A(y),z] = [\proj_A(y),\ang{\proj_A(y),x,z}]
\end{align*}
since $\betw{\proj_A(y)--w--\proj_A(x)--x}$ and similarly for $z$; and
\begin{align*}
[\proj_A(y),x] \cap [\proj_A(y),z]
= [\proj_A(y),\proj_{[x,z]}(\proj_A(y))]
\subseteq [\proj_A(y),y]
\end{align*}
by \cref{thm:med-meet,thm:med-proj}.
Thus $w \in [y, \proj_A(y)]$, so by definition of $\proj_A(y)$, we have $w = \proj_A(y)$, i.e., $\betw{\proj_A(x)--{\proj_A(y)}--{\proj_A(z)}}$.

For uniqueness, note that if $f : X -> \cvx(A)$ is a median homomorphism fixing $\cvx(A)$, then for $x \in X$, for every $a \in A$, we have $\ang{x,f(x),a} \in \cvx(A)$, whence $\ang{x,f(x),a} = f(\ang{x,f(x),a}) = \ang{f(x),f(f(x)),f(a)} = \ang{f(x),f(x),a} = f(x)$, i.e., $f(x) \in [x,a]$; thus $f(x) = \proj_A(x)$.
\end{proof}

\begin{remark}
\label{rmk:homom-cvx}
For a median homomorphism $f : (X,G) -> (Y,H)$:
\begin{itemize}
\item  Preimage $f^{-1}$ is easily seen to preserve convex sets.
\item  For $A \subseteq X$, by \cref{thm:homom-proj}, we have
\begin{align*}
f(\cvx(A)) = f(\proj_A(X)) = \proj_{f(A)}(f(X)) = \cvx(f(A)) \cap f(X).
\end{align*}
Thus if $f$ is surjective, then image under $f$ also preserves convex sets.
\item  Hence, a surjective median homomorphism maps a geodesic onto a geodesic possibly with some repeated vertices, and is in particular a reflexive graph homomorphism.
\end{itemize}
These apply in particular to $\proj_A : X ->> \cvx(A)$.
\end{remark}

\begin{proposition}
\label{thm:proj-proj}
Let $\emptyset \ne A, B \subseteq X$ be two convex sets.
Then
\begin{equation*}
\proj_A \circ \proj_B \circ \proj_A = \proj_A \circ \proj_B = \proj_{\proj_A(B)} : X ->> \proj_A(B).
\end{equation*}
Thus this map and $\proj_B \circ \proj_A$ restrict to inverse graph isomorphisms $\proj_A(B) \cong \proj_B(A)$.
\begin{center}
\begin{tikzpicture}
\draw[pattern=north east lines] (0,0) -- ++(2,0) -- ++(0,2) -- ++(-2,0) -- node[left]{$A$} cycle;
\draw[pattern=north east lines] (6,-1) -- ++(2,0) -- node[right]{$B$} ++(0,2) -- ++(-2,0) -- cycle;
\draw[line width=4pt] (2,0) -- node[right]{$\proj_A(B)$} (2,1);
\node at (4,0.5) {$\cong$};
\draw[line width=4pt] (6,0) -- node[left]{$\proj_B(A)$} (6,1);
\end{tikzpicture}
\end{center}
Moreover, if $A \cap B \ne \emptyset$, then $\proj_A \circ \proj_B = \proj_B \circ \proj_A = \proj_{A \cap B}$.
\end{proposition}

\begin{proof}
By \cref{thm:homom-proj}, we have
$\proj_A \circ \proj_B = \proj_{\proj_A(B)} \circ \proj_A$; thus, composing on the right with another $\proj_A$ or $\proj_B$ yields the same thing.
In particular, $\proj_A \circ \proj_B$ is an idempotent median homomorphism.
Its image is clearly $\proj_A(B)$, which is convex by \cref{rmk:homom-cvx}; thus in fact, the map is $\proj_{\proj_A(B)}$.

If $A \cap B \ne \emptyset$, then $A \cap B$ is fixed by both $\proj_A$ and $\proj_B$, thus contained in $\im(\proj_A \circ \proj_B) = \proj_A(B)$, which is clearly contained in $A$.
To see that it is also contained in $B$: let $c \in A \cap B$; then for any $b \in B$, we have $\proj_A(b) \in [c,b] \subseteq B$.
\end{proof}

\begin{corollary}[Helly's theorem]
\label{thm:helly}
Any finitely many pairwise intersecting nonempty convex sets $A_0, \dotsc, A_{n-1} \subseteq X$ have nonempty intersection.
\end{corollary}
\begin{proof}
All of the projections $\proj_{A_0}, \dotsc, \proj_{A_{n-1}}$ commute; thus applying the composite of all of them in any order to any point yields a point in the intersection (since each $\proj_{A_i}$ could have been applied last).
\end{proof}

\subsection{Hyperplanes and half-spaces}
\label{sec:half}

We continue to let $(X,G)$ be a median graph.

\begin{definition}
\label{def:half}
A \defn{half-space} $H \subseteq X$ is a convex set whose complement is also convex.

Let $\@H_\cvx = \@H_\cvx(X) = \@H_\cvx^G(X) \subseteq 2^X$ denote the set of all half-spaces, and $\@H_\cvx^*(X) := \@H_\cvx(X) \setminus \{\emptyset, X\}$ be the \emph{nontrivial} half-spaces (recall \cref{conv:H}).

For $H \in \@H_\cvx^*$, its inward edge boundary $\partial_\ie H \subseteq G$ is an \defn{(oriented) hyperplane}.
Thus, we have a canonical bijection $\partial_\ie$ between $\@H_\cvx^*$ and the set of hyperplanes in $G$.%
\footnote{Our ``half-spaces'' are called \emph{prime convex sets} in \cite{BHmed}; conventions vary as to whether the trivial half-spaces $\emptyset, X$ are allowed, e.g., in \cite{Isbmed}, \cite{BHmed}, \cite{Rolmed}, \cite{Bowmed}.
The term ``hyperplane'' seems to be used with several meanings which are all isomorphic to each other and to ``half-space'' (or the unoriented version), e.g., the unordered pair $\{H, \neg H\}$ \cite{Rolmed}, or the subspace of the (geometrically realized) CAT(0) cube complex cutting each boundary edge of $H$ at the midpoint \cite{HWspecial}, \cite{Bowmed}.}
\end{definition}

\begin{lemma}
\label{thm:half-edge}
Each edge $(x,y) \in G$ is on a unique hyperplane, namely the inward boundary of
\begin{equation*}
\cone_x y = \proj_{\{x,y\}}^{-1}(y) = \{z \in X \mid d(x,z) > d(y,z)\}.
\end{equation*}
Conversely, each nontrivial half-space $H \subseteq X$ is equal to $\cone_x y$ for any edge $(x,y) \in \partial_\ie H$.
Thus the inverse of the bijection $\partial_\ie$ between half-spaces and hyperplanes is induced by the quotient map
\begin{align*}
G &-->> \@H^*_\cvx(X) \\
(x,y) &|--> \cone_x y,
\end{align*}
i.e., hyperplanes are equivalence classes of edges, called \defn{hyperplane-equivalent}.
\end{lemma}
\begin{proof}
For $(x,y) \in G$, $H := \proj_{\{x,y\}}^{-1}(y)$ is indeed a half-space, being a median preimage of a half-space; and clearly $(x,y) \in \partial_\ie H$.
Conversely, for any nontrivial half-space $H \subseteq X$ and $(x,y) \in \partial_\ie H$, that $H$ is a half-space means precisely that the indicator function $X ->> \{x,y\}$ taking $H$ to $y$ is a median homomorphism, which clearly fixes the convex set $\{x,y\}$, thus is equal to $\proj_{\{x,y\}}$.
\end{proof}

\begin{remark}
For a nontrivial half-space $H \subseteq X$, the inner and outer vertex boundaries $\partial_\iv H, \partial_\ov H$ are easily seen to be the convex sets $\proj_H(\neg H), \proj_{\neg H}(H)$ from \cref{thm:proj-proj}, with the canonical isomorphism between them given by the edge boundary $\partial_\ie H : \partial_\ov H -> \partial_\iv H$.
\end{remark}

\begin{lemma}
\label{thm:half-edge-square}
Hyperplane-equivalence is the equivalence relation on $G$ generated by the pairs of parallel sides of squares (4-cycles).
\end{lemma}
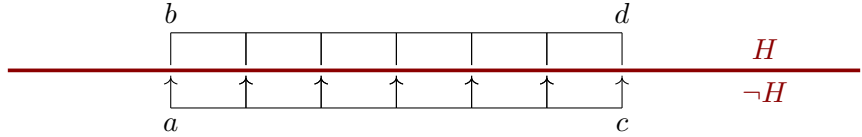
\begin{equation*}
\begin{tikzpicture}
\foreach \i in {0,...,5} {
    \pgfmathsetmacro\j{\i+1}
    \draw[graph edge] (\i,0) -- (\j,0);
    \draw[graph edge] (\i,1) -- (\j,1);
    \draw[graph edge, mid arrow=.4:] (\i,0) -- (\i,1);
    \draw[graph edge, mid arrow=.4:] (\j,0) -- (\j,1);
}
\node[below] at (0,0) {$a$};
\node[above] at (0,1) {$b$};
\node[below] at (6,0) {$c$};
\node[above] at (6,1) {$d$};
\draw[graph cut cross] (-2,.5) -- node[pos=.9,above]{$H$} node[pos=.9,below]{$\neg H$} (9,.5);
\end{tikzpicture}
\end{equation*}
\begin{proof}
Given a square, each vertex is between its neighbors, from which it is easily seen that a hyperplane containing one edge must also contain the parallel edge; thus hyperplane-equivalence contains said equivalence relation.
Conversely, if $(a,b), (c,d) \in \partial_\ie H$ for some $H \in \@H_\cvx^*(X)$, find a geodesic between $a,c$, which must lie in the convex set $\partial_\ov H$, hence be matched via $\partial_\ie H$ to a geodesic between $b,d$, which together with the matching forms the desired strip of squares.
\end{proof}

\begin{definition}
\label{def:nested}
Two half-spaces $H, K \in \@H_\cvx(X)$ are \defn{nested} if $H$ is comparable (under inclusion) with one of $K, \neg K$, or equivalently, one of $H, \neg H$ is disjoint from one of $K, \neg K$.

The \defn{corners} of $H, K$ are $\neg^a H \cap \neg^b K$, where $a, b \in \{0,1\}$; that is, the four sets
\begin{equation*}
H \cap K, H \cap \neg K, \neg H \cap K, \neg H \cap \neg K.
\end{equation*}
Thus $H, K$ are nested iff they have an empty corner.
\end{definition}

(Here and below, $\neg^0 H := H$ and $\neg^1 H := \neg H$.)

\begin{lemma}
\label{thm:nonnested}
Half-spaces $H_0, \dotsc, H_{n-1} \in \@H_\cvx$ are pairwise non-nested iff there exists an embedding $\vec{a} |-> x_{\vec{a}}$ of the Hamming cube graph on $\{0,1\}^n$ such that $x_{\vec{a}} \in \neg^{a_0} H_0 \cap \dotsb \cap \neg^{a_{n-1}} H_{n-1}$.
\end{lemma}
(Recall that the \defn{Hamming cube} on $\{0,1\}^n$ is the graph with edges precisely between any two strings differing in a single bit.)
\begin{proof}
Given such an $n$-cube, the $n$ hyperplanes cutting it are clearly pairwise non-nested.
Conversely, given pairwise non-nested $H_0, \dotsc, H_{n-1} \in \@H_\cvx(X)$, pick $x \in \neg H_0 \cap \dotsb \cap \neg H_{n-1}$, let $x_{\vec1} := \proj_{H_0 \cap \dotsb \cap H_{n-1}}(x) = (\proj_{H_0} \circ \dotsb \circ \proj_{H_{n-1}})(x)$, and let $x_{\vec{a}} := \proj_{\neg^{a_0} H_0 \cap \dotsb \cap \neg^{a_{n-1}} H_{n-1}}(x_{\vec1})$ for each $\vec{a} \in \{0,1\}^n$; using \cref{thm:proj-proj}, this is easily seen to be an $n$-cube.
\end{proof}

\begin{corollary}
\label{thm:nested-tree}
$G$ is a tree iff all of its half-spaces are pairwise nested.
\end{corollary}
\begin{proof}
$\Longrightarrow$:
Two non-nested half-spaces yield a 4-cycle by the preceding lemma.

$\Longleftarrow$:
A cycle has two distinct edges on a hyperplane, whence there is a 4-cycle by \cref{thm:half-edge-square}.
\end{proof}

\begin{definition}
\label{def:half-cover}
$H \in \@H_\cvx(X)$ is a \defn{successor} of $K \in \@H_\cvx(X)$ if $H \supsetneq K$ but there is no half-space strictly in between, or equivalently, $H \cup \neg K = X$, $H \cap \neg K \ne \emptyset$, and $H, \neg K$ are each minimal as such.
\end{definition}

\begin{lemma}
For $H, K \in \@H_\cvx^*$, $H$ is a successor of $K$ iff $H \supseteq K$ and $\partial_\iv H \cap \partial_\ov K \ne \emptyset$.
\end{lemma}
\begin{proof}
$\Longrightarrow$ is clear.
Conversely, suppose $H \supseteq K$ and $x \in \partial_\iv H \cap \partial_\ov K$; then clearly $H \supsetneq K$, and any $H \supseteq L \supseteq K$ must either contain $x$ but not $\proj_{\neg H}(x) \in \partial_\ov H$, hence equal $H$, or contain $\proj_K(x) \in \partial_\iv K$ but not $x$, hence equal $K$.
\end{proof}

\begin{corollary}
\label{thm:hyp-adj}
For $H, K \in \@H_\cvx^*$, we have $\partial_\v H \cap \partial_\v K \ne \emptyset$ iff $H, K$ are equal, or complementary, or non-nested, or one of $H, \neg H$ is a successor of one of $K, \neg K$, i.e., either none of $H, \neg H$ is comparable with $K, \neg K$, or one pair is comparable but there is no third half-space strictly in between.
\end{corollary}
If these hold, we say the hyperplanes $\partial_\ie H, \partial_\ie K$ are \defn{adjacent}.
\begin{proof}
By flipping $H, K$ if necessary, we may assume there is $x \in \partial_\iv H \cap \partial_\ov K \subseteq H \cap \neg K$.
If $\proj_K(x) \in \neg H$ or $\proj_{\neg H}(x) \in K$, then these points are equal and $H = \neg K = \cone_{\proj_K(x)}{x}$.
Otherwise, $H \cap K \ne \emptyset$ and $\neg H \cap \neg K \ne \emptyset$.
If also $\neg H \cap K \ne \emptyset$, then $H, K$ are non-nested.
Otherwise, $K \subseteq H$ and $x \in \partial_\iv H \cap \partial_\ov K$, so $H$ is a successor of $K$.
\end{proof}

\begin{lemma}
\label{thm:half-dist}
For any $x, y \in X$, $d(x,y) = \abs{\{H \in \@H_\cvx(X) \mid x \not\in H \ni y\}}$.
\end{lemma}
\begin{proof}
Let $x = x_0 \mathrel{G} \dotsb \mathrel{G} x_n = y$ be any geodesic, so $n = d(x,y)$.
Then each half-space $x \not\in H \ni y$ has $x_i \not\in H \ni x_{i+1}$ for a unique $i < n$, so we get a bijection between the set of such $H$ and $n$.
\end{proof}

\begin{lemma}
\label{thm:half-sep}
Two disjoint convex sets $A, B \subseteq X$ can be separated by a half-space $A \subseteq H \subseteq \neg B$.
\end{lemma}
\begin{proof}
If $A$ or $B$ is empty, this is trivial; so suppose not.
Pick a geodesic $A \ni x_0 \mathrel{G} \dotsb \mathrel{G} x_n \in B$, where $n = d(A, B)$.
Then $H := \cone_{x_1}{x_0}$ works.
Indeed, note that $x_0 = \proj_A(x_n)$ (being the closest point in $A$ to $x_n$), whence $A \subseteq \cone_{x_n}{x_0} \subseteq \cone_{x_1}{x_0}$ by \cref{eq:int-assoc}, and similarly $B \subseteq \cone_{x_0}{x_n} \subseteq \cone_{x_0}{x_1}$.
\end{proof}

\begin{corollary}
\label{thm:cvx-finite}
Every interval $[x,y]$ in a median graph is finite.
More generally, a convex hull $\cvx(\{x_0,\dotsc,x_n\})$ of finitely many points is finite.
\end{corollary}
\begin{proof}
By \cref{thm:half-dist}, there are finitely many half-spaces $H \subseteq [x,y]$; by \cref{thm:half-sep}, a vertex $z \in [x,y]$ is determined by the half-spaces containing it, whence $[x,y]$ is finite.
By the proof of \cref{thm:proj}, $\cvx(\{x_0,\dotsc,x_n\})$ consists of points of the form $\ang{x,x_n,\ang{x,x_{n-1},\ang{\dotsb\ang{x,x_2,\ang{x,x_1,x_0}}}}}$ which are in an interval between $x_n$ and a point in an interval between $x_{n-1}$ and \ldots.
\end{proof}

\begin{corollary}
\label{thm:half-full}
For two median graphs $(X,G), (Y,H)$, a function $f : X -> Y$ is a median homomorphism iff preimage $f^{-1}$ preserves convex sets (or even just half-spaces).
\end{corollary}
\begin{proof}
$\Longrightarrow$ follows from \cref{rmk:homom-cvx}.
Conversely, suppose $f^{-1}$ preserves half-spaces, and let $x, y, z \in X$ with $f(x) \not\in [f(y),f(z)]$; then there is a half-space $H \subseteq Y$ containing $f(x)$ but not $f(y), f(z)$, whence $x \in f^{-1}(H) \not\ni y, z$, whence $x \not\in [y,z]$.
\end{proof}

\subsection{Pocsets and duality}
\label{sec:duality}

The half-spaces $\@H_\cvx^G(X)$ of a median graph $(X,G)$ have the following structure:

\begin{definition}
A \defn{pocset} $P = (P,{\le},{\neg},0)$ is a poset equipped with an order-reversing involution and a least element such that $0 \ne \neg 0$, and for any $p \in P$, $0$ is the only lower bound of $p, \neg p$.%
\footnote{We use the term ``pocset'' from \cite{Rolmed}; other terms include \emph{proset} \cite{Bowmed} and \emph{binary message} \cite{Isbmed}, with differing conventions as to whether $0$ and its complement are included.}

A \defn{profinite pocset} is a pocset equipped with a compact topology making $\neg$ continuous and satisfying the \defn{Priestley separation axiom}%
\footnote{also known as \emph{totally order-disconnected}; see e.g., \cite{Jstone}}%
: for any $p \not\le q$, there exists a clopen upward-closed $U \subseteq P$ containing $p$ but not $q$.
This easily implies that the topology is Hausdorff zero-dimensional, and that the partial order $\le$ is a closed set in $P^2$.
\end{definition}

\begin{example}
$2 = \{0,1\}$ is easily seen to be a profinite pocset, whence so are its powers $2^X$, whence so are the closed subpocsets of $2^X$; this includes $\@H_\cvx(X)$ for a median graph $(X,G)$.

For a median homomorphism $f : (X,G) -> (Y,H)$, taking preimage under $f$ is easily seen to yield a continuous pocset homomorphism $f^* : \@H_\cvx(Y) -> \@H_\cvx(X)$, thereby making $(X,G) |-> \@H_\cvx(X)$ into a contravariant functor.
\end{example}

\begin{lemma}
Every nontrivial half-space $H \in \@H_\cvx^*(X)$ is an isolated point in $\@H_\cvx(X) \subseteq 2^X$.
\end{lemma}
\begin{proof}
For any $(x,y) \in \partial_\ie H$, by \cref{thm:half-edge}, $H$ is the unique half-space containing $y$ but not $x$.
\end{proof}

\begin{theorem}[Isbell \cite{Isbmed}, Werner \cite{Werdual}]
\label{thm:duality}
We have a contravariant equivalence of categories
\begin{align*}
\brace*{\begin{gathered}\text{median graphs,}\\\text{median homomorphisms}\end{gathered}} &--> \brace*{\begin{gathered}\text{profinite pocsets with every nontrivial element isolated,}\\\text{continuous pocset homomorphisms}\end{gathered}} \\
(X,G) &|--> \@H_\cvx(X) \\
f : (X,G_X) -> (Y,G_Y) &|--> \paren*{\begin{aligned}
    f^* : \@H_\cvx(Y) &-> \@H_\cvx(X) \\
    H &|-> f^{-1}(H)
\end{aligned}}.
\end{align*}
\end{theorem}

We devote the rest of this subsection to constructing the inverse of this functor.

\begin{definition}
\label{def:orientation}
Let $P$ be a pocset.
An \defn{orientation}%
\footnote{\label{ftn:orientation}%
We borrow ``orientation'' from \cite{Rolmed}, which, however, does not include the upward-closure condition.
That work calls upward-closed orientations \emph{ultrafilters}, while noting that these are weaker than the usual usage of that term in Boolean algebras, and in particular, \emph{need not be filters in the usual order-theoretic sense}.
We have chosen our present terminology since we will also be dealing extensively with Boolean ultrafilters (see \cref{sec:ends}), and since we will never need non-upward-closed orientations.
The equivalent term \emph{flow} is used in \cite{Bowmed}.}
on $P$ is an upward-closed subset $U \subseteq P$ containing exactly one of $p, \neg p$ for each $p \in P$.
This last condition means that, letting $\neg(U)$ denote the image of $U$ under $\neg : P -> P$ while $\neg U := P \setminus U$ denotes the complement, we have
\begin{equation*}
\neg(U) = \neg U.
\end{equation*}
If only $\subseteq$ holds, we call $U$ a \defn{partial orientation}.
We let $\@U(P) \subseteq 2^P$ denote the set of orientations.

If $P$ is a topological pocset, then note that the \defn{clopen orientations} $U \subseteq P$ are equivalently the closed or open orientations, since $\neg : P -> P$ is a homeomorphism.
We denote the set of these by $\@U^\circ(P) \subseteq \@U(P)$.

Two orientations $U, V \subseteq P$ are \defn{adjacent} if $\abs{U \setminus V} = 1$ (equivalently, $\abs{V \setminus U} = 1$).
\end{definition}

Intuitively, if $P$ is thought of as a collection of ``abstract half-spaces'', then an orientation is a consistent and complete description of where a ``point'' lies in relation to each ``half-space'' in $P$.

\begin{example}
For a median graph $(X,G)$, each vertex $x \in X$ determines a \defn{principal orientation} $\^x := \{H \in \@H_\cvx(X) \mid x \in H\} \subseteq \@H_\cvx(X)$, which is clopen.
Moreover, by \cref{thm:half-sep}, $\bigcap \^x = x$, while by \cref{thm:half-dist}, $x \mathrel{G} y$ iff $\^x, \^y$ are adjacent, so that
\begin{equation*}
\begin{aligned}
X &--> \@U^\circ(\@H_\cvx(X)) \\
x &|--> \^x
\end{aligned}
\end{equation*}
is a graph embedding.
\end{example}

\begin{proposition}
\label{thm:median-dual}
For any median graph $(X,G)$, the above embedding $x |-> \^x$ is an isomorphism.
\end{proposition}
\begin{proof}
We must show that every clopen orientation $U \subseteq \@H_\cvx(X)$ is $\^x$ for some $x \in X$.
Since $U$ is clopen, there is a finite $A \subseteq X$, which we may assume to be convex by \cref{thm:cvx-finite}, such that for every $H \in \@H_\cvx(X)$, $H \in U$ iff there is $K \in U$ with $H \cap A = K \cap A$.
If $K \cap A = \emptyset$ for some $K \in U$, then $\emptyset \in U$, a contradiction; thus $A$ intersects every element of $U$.
Also, every $H, K \in U$ have $H \cap K \ne \emptyset$, or else $H \subseteq \neg K \implies \neg K \in U$.
Thus by \cref{thm:helly}, there is $x \in \bigcap U \cap A$, whence $U \subseteq \^x$, whence $U = \^x$ since both are orientations.
\end{proof}

To complete the proof of \cref{thm:duality}, it remains to show that starting from an abstract profinite pocset $P$ with every nontrivial element isolated, $\@U^\circ(P)$ is a median graph with $\@H_\cvx(\@U^\circ(P)) \cong P$.

\begin{lemma}
\label{thm:orient-sep}
Let $P$ be a profinite pocset.
Then every closed partial orientation $F \subseteq P$ extends to a clopen orientation $F \subseteq U \subseteq P$.
\end{lemma}
\begin{proof}
First, note that compactness together with the Priestley separation axiom implies that a closed upward-closed $F \subseteq P$ disjoint from a closed downward-closed $D \subseteq P$ may be separated by a clopen upward-closed $F \subseteq U \subseteq \neg D$, by a standard finite covering argument.

Now for a closed partial orientation $F \subseteq P$, we have $F \subseteq \neg \neg(F)$, whence there is a clopen upward-closed $F \subseteq U \subseteq \neg \neg(F)$, whence also $F = \neg(\neg(F)) \subseteq \neg(\neg U) = \neg \neg(U)$, whence $F$ is contained in the clopen partial orientation $U \cap \neg \neg(U)$.
Thus we may assume to begin with that $F$ is clopen.

Taking $F = \{q \in P \mid q \ge p\}$ shows that each $0 < p \in P$ is in a clopen partial orientation.

Thus, we may assume $\neg 0 \in F$, i.e., $0 \not\in \neg \neg(F)$.
So each $p \in \neg \neg(F)$ is in a clopen partial orientation, and so by compactness $\neg \neg(F)$ is contained in a finite union of such sets $G_0, \dotsc, G_{n-1} \subseteq P$.
Now put $F_0 := F$ and $F_{i+1} := F_i \cup (G_i \cap \neg \neg(F_i))$; then $U := F_n$ works.
\end{proof}

\begin{proposition}
\label{thm:poc-dual}
Let $P$ be a profinite pocset with every nontrivial element isolated.
Then the adjacency graph on $\@U^\circ(P)$ is a median graph, whose:
\begin{itemize}
\item  Neighbors of $U \in \@U^\circ(P)$ are $U \triangle \{p, \neg p\}$ for each minimal $p \in U \setminus \{\neg 0\}$.
\item  Path metric $d(U,V) = \abs{U \setminus V} = \abs{V \setminus U}$.
\item  Intervals $W \in [U,V] \iff U \cap V \subseteq W \iff W \subseteq U \cup V \iff (U \setminus W) \cup (W \setminus V) \subseteq U \setminus V$.
\item  Medians $\ang{U,V,W} = \{p \in P \mid p \text{ belongs to at least two of } U, V, W\}$.
\end{itemize}
\end{proposition}
\begin{proof}
There is at least one vertex, by \cref{thm:orient-sep} applied to $\{\neg 0\}$.

By definition, $U, V$ are neighbors iff $U \setminus V$ is a single vertex, which must then be minimal in $U$ since $V$ was upward-closed.
Conversely, for minimal $p \in U \setminus \{\neg 0\}$, $V := U \triangle \{p, \neg p\}$ is still an orientation (since $p \in U$ is minimal and $\neg p \in \neg(U) = \neg U$ is maximal), and clopen since $p$ is isolated.

It follows that $d(U,V) = \abs{U \setminus V}$: $\ge$ because a path may modify one element at a time; $\le$ because we may modify $U$ to $V$ in $\abs{U \setminus V}$ steps by flipping minimal elements.
In particular, note that $d(U,V) < \infty$, i.e., the graph is connected, since $U \setminus V$ is a compact set of isolated points.

The conditions $U \cap V \subseteq W$ and $W \subseteq U \cup V$ are equivalent since $U, V, W$ are orientations (take image under $\neg : P -> P$ followed by complement).
The condition $(U \setminus W) \cup (W \setminus V) \subseteq U \setminus V$ is equivalent to $U \setminus W \subseteq \neg V$ and $W \setminus V \subseteq U$, i.e., both of the former conditions.
This third condition is equivalent to $d(U,W) + d(W,V) = d(U,V)$, i.e., $W \in [U,V]$.

Thus, $M \in [U,V] \cap [V,W] \cap [W,U]$ iff $(U \cap V) \cup (V \cap W) \cup (W \cap U) \subseteq M \subseteq (U \cup V) \cap (V \cup W) \cap (W \cup U)$ iff $M = \{p \in P \mid p \text{ belongs to at least two of } U, V, W\}$; and this set is easily seen to still be a clopen orientation if $U, V, W$ are.
\end{proof}

\begin{proposition}
\label{thm:pocset-dual}
Let $P$ be a profinite pocset with every nontrivial element isolated.
Then
\begin{align*}
P &--> \@H_\cvx(\@U^\circ(P)) \\
p &|--> \^p := \{U \in \@U^\circ(P) \mid p \in U\}
\end{align*}
is a topological pocset isomorphism.
\end{proposition}
\begin{proof}
It is an order-embedding, since for $p \not\le q \in P$, $F := \{r \in P \mid p \le r \OR \neg q \le r\}$ is a closed partial orientation, hence extends by \cref{thm:orient-sep} to $U \in \@U^\circ(P)$ such that $U \in \^p \setminus \^q$.
To see that it is surjective: take a nontrivial half-space $H \subseteq \@U^\circ(P)$ and an edge $(U,V) \in \partial_\ie H$; then by \cref{thm:poc-dual}, $V \setminus U = \{p\}$ is a singleton, and so by \cref{thm:half-edge}, $H = \cone_U V$ which by \cref{thm:poc-dual} is the set of $W \in \@H_\cvx(\@U^\circ(P))$ such that $V \setminus U \subseteq W$, i.e., $H = \^p$.
\end{proof}

\begin{corollary}
Let $P, Q$ be profinite pocsets with every nontrivial element isolated, $g : P -> Q$ be a continuous pocset homomorphism.
Then
\begin{align*}
g^* : \@U^\circ(Q) &--> \@U^\circ(P) \\
U &|--> g^{-1}(U)
\end{align*}
is a median homomorphism.
\end{corollary}
\begin{proof}
By \cref{thm:half-full}, this is equivalent to $(g^*)^{-1} : 2^{\@U^\circ(P)} -> 2^{\@U^\circ(Q)}$ mapping half-spaces to half-spaces.
Indeed, the composite of $(g^*)^{-1}$ with the canonical isomorphism $P \cong \@H_\cvx(\@U^\circ(P)) \subseteq 2^{\@U^\circ(P)}$ is easily seen to be $g$ composed with $Q \cong \@H_\cvx(\@U^\circ(Q))$.
\end{proof}

This completes the construction of the inverse functor in \cref{thm:duality}.

\begin{remark}
\Cref{thm:duality} is a special case of the duality between \emph{median algebras} and profinite pocsets (not necessarily with nontrivial elements isolated) due to Isbell~\cite{Isbmed} and Werner~\cite{Werdual}.
This duality in turn belongs to the general theory of Stone-type dualities induced by homomorphisms to $2 = \{0,1\}$; see \cite[VI~\S3]{Jstone}.
\end{remark}

\subsection{Finiteness conditions on median graphs and pocsets}
\label{sec:finiteness}

\begin{corollary}
A median graph is finite iff it has finitely many half-spaces.
\end{corollary}
\begin{proof}
$\Longrightarrow$ is obvious; $\Longleftarrow$ by \cref{thm:median-dual}.
\end{proof}

\begin{corollary}
\label{thm:med-hyp-fin}
A median graph $(X,G)$ has finite hyperplanes iff each half-space is non-nested with finitely many others.
\end{corollary}
\begin{proof}
Half-spaces non-nested with $H \in \@H_\cvx(X)$ correspond to nontrivial half-spaces of $\partial_\iv H$, by \cref{thm:nonnested}.
\end{proof}

\begin{corollary}
\label{thm:med-hyp-locfin}
For a median graph $(X,G)$, the following are equivalent:
\begin{enumerate}[label=(\roman*)]
\item \label{thm:med-hyp-locfin:adj}
The adjacency graph (from \cref{thm:hyp-adj}) on hyperplanes of $G$ is locally finite.
\item \label{thm:med-hyp-locfin:nest-cover}
Each nontrivial half-space is non-nested with or a successor of finitely many others.
\item \label{thm:med-hyp-locfin:hyp-locfin}
$G$ is locally finite, and all hyperplanes are finite.
\end{enumerate}
\end{corollary}
\begin{proof}
We easily have \cref{thm:med-hyp-locfin:adj}$\iff$\cref{thm:med-hyp-locfin:nest-cover} (using that $H$ is a successor of $K$ iff $\neg K$ is a successor of $\neg H$).
By the preceding corollary, it thus remains to show, assuming $G$ has finite hyperplanes, that \cref{thm:med-hyp-locfin:adj} iff $G$ is locally finite.

Suppose \cref{thm:med-hyp-locfin:adj} holds.
Then $G$ is locally finite since the neighbors of $x$ correspond to half-spaces containing $x$ on the inner boundary, which correspond to a clique of adjacent hyperplanes by \cref{thm:hyp-adj}.

Now suppose \cref{thm:med-hyp-locfin:adj} fails; let $\partial_\ie H$ for $H \in \@H_\cvx^*(X)$ be adjacent to infinitely many other hyperplanes.
Then some $x \in \partial_\v H$ belongs to infinitely many other hyperplanes, whence $x$ has infinite degree.
\end{proof}

\begin{remark}
Without assuming finite hyperplanes, the above shows that if the adjacency graph on hyperplanes of $G$ has no infinite cliques, then $G$ is locally finite.
The converse is false.
Consider two consecutive edges, the second of which is glued to one end of a strip of two squares, the second of which is glued to one end of a stack of two cubes, etc.
There are infinitely many pairwise non-nested half-spaces; but the ``common points on their boundaries'' are ``median ends'' rather than vertices (see \cref{def:med-ends}).
\end{remark}

\subsection{Subpocsets and wallspaces}
\label{sec:subpoc}

We now describe ways in which \cref{thm:duality} can be applied to construct a median graph, as the dual of a pocset satisfying the required conditions.
The most common examples of pocsets are subpocsets of powersets $2^X$.
The following characterizes when these yield median graphs:

\begin{proposition}
\label{thm:subpoc-isolated}
Let $X$ be a set, $\@H \subseteq 2^X$ be a subpocset.
The following are equivalent:
\begin{enumerate}[label=(\roman*)]
\item \label{thm:subpoc-isolated:isolated}
$\@H$ is closed (in the Cantor space $2^X$) and every nontrivial element of $\@H$ is isolated.
\item \label{thm:subpoc-isolated:finsep}
$\@H$ is \defn{finitely separating}: for every $x, y \in X$, there are only finitely many $H \in \@H$ containing $x$ but not $y$.
\end{enumerate}
\end{proposition}
\begin{proof}
\cref{thm:subpoc-isolated:isolated}$\implies$\cref{thm:subpoc-isolated:finsep} follows from \cref{thm:poc-dual} applied to the principal orientations $\^x, \^y \in \@U^\circ(\@H)$.
Conversely, assume \cref{thm:subpoc-isolated:finsep}, let $A \in 2^X \setminus \{\emptyset, X\}$, and let $x \in A \not\ni y$.
For every $H \in \@H \setminus \{A\}$ such that $x \in H \not\ni y$, we have either some $x_H \in A \setminus H$ or some $y_H \in H \setminus A$.
Then the set of all $B \in 2^X$ containing $x$ and each $x_H$ but not $y$ or any $y_H$ is a clopen neighborhood of $A$ disjoint from $\@H \setminus \{A\}$.
Thus the limit points of $\@H$ are $\subseteq \{\emptyset, X\}$, proving \cref{thm:subpoc-isolated:isolated}.
\end{proof}

\begin{definition}
\label{def:walls}
A \defn{wallspace} \cite{Nwalls}, \cite{CNwalls} is a set $X$ equipped with an arbitrary subpocset $\@H = \@H(X) \subseteq 2^X$ satisfying the equivalent conditions of \cref{thm:subpoc-isolated}; we call $\@H(X)$ the \defn{walling} of the wallspace $X$.
Per \cref{conv:H}, we also put $\@H^*(X) := \@H(X) \setminus \{\emptyset,X\}$.
\end{definition}

Thus, for a median graph $(X,G)$, we have a canonical choice of walling, namely $\@H_\cvx(X)$; while from an \emph{arbitrary} walling $\@H(X)$, we may construct by \cref{thm:duality} the dual median graph $\@U^\circ(\@H(X))$ equipped with the principal orientations map $x |-> \^x$ from $X$, which may be thought of as a ``median completion'' of $X$.

Note however that this ``completion'' need not extend $X$ because the map $x |-> \^x$ is not injective in general.
We call the preimage of a principal orientation an \defn{$\@H$-block}; in other words, $\@H$-blocks are equivalence classes of points contained in exactly the same sets in $\@H$.
We denote the set of $\@H$-blocks by $X/\@H$, and the $\@H$-block containing $x \in X$ by $[x]_{\@H}$.

In fact, when $\@H$ above is a subset of the half-spaces of some median graph on $X$, this ``median completion'' of $(X,\@H)$ is just a quotient of $X$:

\begin{proposition}
\label{thm:median-subpoc}
Let $(X,G)$ be a median graph, $\@H \subseteq \@H_\cvx(X)$ be a subpocset of half-spaces.
Then $\@H$ is a walling, and the median homomorphism $X \cong \@U^\circ(\@H_\cvx(X)) -> \@U^\circ(\@H)$ induced by the inclusion $\@H `-> \@H_\cvx(X)$ is surjective.
Hence, $\@U^\circ(\@H)$ may be constructed up to isomorphism as the set of $\@H$-blocks, equipped with the $G$-adjacency graph.
\end{proposition}
\begin{proof}
Since $\@H_\cvx(X) \setminus \@H$ consists only of isolated points, $\@H \subseteq \@H_\cvx(X)$ is closed.
So given $U \in \@U^\circ(\@H)$, the upward-closure of $U$ in $\@H_\cvx(X)$ is closed by compactness of $\{(H,K) \mid H \subseteq K\} \subseteq \@H_\cvx(X)^2$, and a partial orientation since $U$ is, hence extends by \cref{thm:orient-sep} to a clopen orientation $U \subseteq V \subseteq \@H_\cvx(X)$, so that $U \subseteq V \cap \@H$, whence $U = V \cap \@H$ since both are orientations.
This shows that $\@U^\circ(\@H_\cvx(X)) -> \@U^\circ(\@H)$ is surjective.

The composite $X \cong \@U^\circ(\@H_\cvx(X)) ->> \@U^\circ(\@H)$ is given by $x |-> \^x \cap \@H$, which identifies two vertices iff they belong to the same half-spaces in $\@H$, i.e., the same $\@H$-block.
If two $\@H$-blocks $[x]_{\@H}, [y]_{\@H}$ are joined by an edge, then that means the interval $[[x]_{\@H}, [y]_{\@H}]$ consists of only $\{[x]_{\@H}, [y]_{\@H}\}$, whence $[x,y] \subseteq [x]_{\@H} \cup [y]_{\@H}$, so any geodesic from $x$ to $y$ must contain an edge between $[x]_{\@H}, [y]_{\@H}$.
\end{proof}

\begin{remark}
This says that embeddings on the right side of the duality \ref{thm:duality} correspond to surjections on the left.
Conversely, it is clear that a surjective median homomorphism $f : (X,G) ->> (Y,H)$ between median graphs induces a topological pocset embedding $f^* : \@H_\cvx(Y) `-> \@H_\cvx(X)$.
\end{remark}

\subsection{Ends}
\label{sec:ends}

\begin{definition}
\label{def:ends}
Let $(X,G)$ be a connected graph.
Let $\@H_\fin(X) = \@H_\fin^G(X) \subseteq 2^X$ be the Boolean algebra of all sets with finite boundary.
The \defn{end compactification} $\^X = \^X^G$ is the Stone space of ultrafilters on $\@H_\fin(X)$; thus, a clopen set in $\^X$ is the set $\^A$ of all ultrafilters containing some $A \in \@H_\fin(X)$.
We identify each $x \in X$ with the corresponding principal ultrafilter $\^x \in \^X$; these are dense in $\^X$.
The nonprincipal ultrafilters are called \defn{ends} of $(X,G)$.%
\footnote{While the definition makes sense regardless, if $G$ is not locally finite, e.g., if $G = K_\infty$, $x |-> \^x$ may not be injective.}

For another connected graph $(Y,H)$ and a map $f : X -> Y$, the \defn{induced map} is
\begin{align*}
\^f := (f^*)^{-1} : \^X &--> \^Y \\
U &|--> \{A \in \@H_\fin(Y) \mid f^{-1}(A) \in U\}.
\end{align*}
The induced map exists iff $f^*$ preserves sets with finite boundary, e.g., if $f$ is an injective (or more generally finite-to-one) graph homomorphism.
\end{definition}

See \cite{DKends} for general information on ends of graphs, including the equivalence in the locally finite case between the point-set topological approach we have adopted (dating back to Freudenthal \cite{Freends} and Hopf \cite{Hopends}) and the more common approach via rays (due to Halin \cite{Halends}).

\begin{definition}
\label{def:med-ends}
For a median graph $(X,G)$, the space $\@U(\@H_\cvx(X)) \subseteq 2^{\@H_\cvx(X)}$ of \emph{all} (not necessarily clopen) orientations is called the \defn{profinite median completion} of $(X,G)$.%
\footnote{The nonprincipal orientations $\@U(\@H_\cvx(X)) \setminus \@U^\circ(\@H_\cvx(X))$ are called the \defn{Roller boundary}; see \cite[\S11.12]{Bowmed}.}
\end{definition}

\begin{lemma}
\label{thm:med-fin-cvx}
For a median graph $(X,G)$ with finite hyperplanes, $\@H_\fin(X) \subseteq 2^X$ is the Boolean subalgebra generated by $\@H_\cvx(X)$.
\end{lemma}
\begin{proof}
Let $A \in \@H^*_\fin$; we must show that $A$ is a finite Boolean combination of half-spaces.
By \cref{thm:cvx-finite}, $Y := \cvx(\partial_\v A) \subseteq X$ is finite; thus (by \cref{thm:half-sep}) $A \cap Y$ is a finite Boolean combination of half-spaces in $Y$, whence $\proj_Y^{-1}(A \cap Y)$ is a finite Boolean combination of half-spaces in $X$.
We claim that $A = \proj_Y^{-1}(A \cap Y)$.
The intersections of both with $Y$ are clearly the same; thus it suffices to show that for $x \in X \setminus Y$, $x \in A \iff x \in \proj_Y^{-1}(A \cap Y)$.
Indeed, if $x \in A$, then $\proj_Y(x) \in A \cap Y$, or else the last vertex in $A$ on a geodesic from $x$ to $\proj_Y(x)$ would be on the inner boundary of $A$ but not in $Y$, contradicting the definition of $Y$.
Similarly if $x \not\in A$.
\end{proof}

\begin{proposition}
\label{thm:med-ends}
For a median graph $(X,G)$ with finite hyperplanes, we have a homeomorphism
\begin{align*}
\^X &--> \@U(\@H_\cvx(X)) \\
U &|--> U \cap \@H_\cvx(X).
\end{align*}
\end{proposition}
\begin{proof}
Injectivity: for $U, V \in \^X$, the set of $A \in \@H_\fin(X)$ such that $A \in U \iff A \in V$ is a Boolean algebra, hence if it contains $\@H_\cvx(X)$, must be all of $\@H_\fin(X)$ by the above lemma.
Surjectivity: $V \in \@U(\@H_\cvx(X))$ has the finite intersection property by \cref{thm:helly}, hence extends to an ultrafilter $V \subseteq U \subseteq \@H_\fin(X)$ such that $V = U \cap \@H_\cvx(X)$.
\end{proof}

\begin{remark}
The above homeomorphism is clearly compatible with the respective inclusions of $X$, in that the following triangle commutes:
\begin{equation*}
\begin{tikzcd}
X \dar[hook, "x \mapsto \^x"'] \drar[hook, "x \mapsto \^x"] \\
\^X \rar["\cong"] & \@U(\@H_\cvx(X))
\end{tikzcd}
\end{equation*}
\end{remark}

\begin{lemma}
\label{thm:med-ends-half-basis}
For a locally finite median graph $(X,G)$ with finite hyperplanes, each end has a neighborhood basis of half-spaces.
\end{lemma}
\begin{proof}
Let $U \in \^X \setminus X$ be an end, $\^A \ni U$ be a clopen neighborhood where $A \in \@H_\fin(X)$.
By \cref{thm:cvx-finite}, $\cvx(\partial_\ov A) \subseteq X$ is finite.
By \cref{thm:med-fin-cvx} (applied to it and each point on its outer boundary), it is a finite intersection of half-spaces.
One such half-space $H$ must not contain $U$, since $U \not\in \cvx(\partial_\ov A)$.
Then $U \in \neg \^H$, whence $A \cap \neg H \ne \emptyset$, whence $\neg H \subseteq A$ since $\partial_\ov A \cap \neg H = \emptyset$.
\end{proof}

\begin{corollary}
\label{thm:walls-ends-dense}
For a wallspace $X$ such that $\@U^\circ(\@H(X))$ has finite hyperplanes, each end $U \in \^{\@U^\circ(\@H(X))} \cong \@U(\@H(X))$ is a limit of points in $X$ (not just in $\@U^\circ(\@H(X))$).
\end{corollary}
\begin{proof}
The half-space neighborhoods of $U$ form a filter base in $X$ converging to $U$.
\end{proof}

\subsection{Quasi-isometries and coarse equivalences}
\label{sec:coarse}

We now review some notions from metric geometry.

\begin{definition}
A \defn{proper pseudometric} is one whose balls are all finite.
\end{definition}

\begin{example}
The path metric on a connected locally finite graph is a proper metric.
\end{example}

\begin{definition}
\label{def:coarse}
Let $X, Y$ be pseudometric spaces.
A \defn{bornologous map} $f : X -> Y$ is one with
\begin{equation*}
\forall R < \infty\, \exists S < \infty\, \forall x, x' \in X\, (d_X(x,x') \le R \implies d_Y(f(x),f(x')) \le S).
\end{equation*}
A \defn{quasi-inverse} of $f$ is a map $g : Y -> X$ with uniform distance $d(1_X, g \circ f), d(1_Y, f \circ g) < \infty$.
If $f$ is a bornologous map with a bornologous quasi-inverse, then $f$ is a \defn{coarse equivalence}; if such $f$ exists, then $X, Y$ are \defn{coarsely equivalent}.
We say that $f$ is a \defn{coarse embedding} if it is a coarse equivalence onto its image.
This is easily seen to be equivalent to: $f$ is bornologous, and moreover
\begin{equation*}
\forall S < \infty\, \exists R < \infty\, \forall x, x' \in X\, (d_Y(f(x), f(x')) \le S \implies d_X(x, x') \le R).
\end{equation*}
Note that the above $\forall\exists\forall$ conditions are equivalent to, respectively:
\begin{gather*}
\forall R < \infty\, \exists S < \infty\, \forall A \subseteq X\, (\diam(A) \le R \implies \diam(f(A)) \le S), \\
\forall S < \infty\, \exists R < \infty\, \forall A \subseteq X\, (\diam(f(A)) \le S \implies \diam(A) \le R).
\end{gather*}

A \defn{quasi-isometry} is a coarse equivalence $f$ such that $S$ above can be taken to be a linear function of $R$, and such that $f$ has a quasi-inverse obeying the same condition.
\end{definition}

See \cite{Roecoarse} or \cite{Drutu-Kapovich:GGT} for background on large-scale geometric notions such as these.

\begin{lemma}
A coarse embedding between proper pseudometric spaces is finite-to-one.
Indeed, diameters of fibers of such a map are uniformly bounded.
\end{lemma}
\begin{proof}
Let $f : X -> Y$ be such a map, and $R < \infty$ with $d(f(x), f(x')) \le 0 \implies d(x, x') \le R$.
\end{proof}

\begin{lemma}
\label{thm:ends-coarse}
Let $(X,G), (Y,H)$ be connected locally finite graphs.
\begin{enumerate}[label=(\alph*)]
\item \label{thm:ends-coarse:induce}
For a coarse embedding $f : X -> Y$ and $A \in \@H_\fin(Y)$, $\diam(\partial_\v f^{-1}(A))$ is uniformly bounded in terms of $\diam(\partial_\v A)$.
In particular, $f^{-1}(A) \in \@H_\fin(X)$, i.e., $f$ induces a map $\^f : \^X -> \^Y$.
\item \label{thm:ends-coarse:bddist}
For coarse embeddings $f, g : X -> Y$ with $d(f, g) < \infty$, for $A \in \@H_\fin(Y)$, $\diam(f^{-1}(A) \triangle g^{-1}(A))$ is uniformly bounded in terms of $\diam(\partial_\v A)$.
In particular, $\^f, \^g : \^X -> \^Y$ agree on ends.
\item \label{thm:ends-coarse:coarseq}
Thus, a coarse equivalence $f : X -> Y$ induces a homeomorphism on ends $\^f : \^X \setminus X \cong \^Y \setminus Y$.
\end{enumerate}

\end{lemma}
\begin{proof}
\cref{thm:ends-coarse:induce}
Let $S < \infty$ such that $d(x, x') \le 1 \implies d(f(x), f(x')) \le S$.
For $A \subseteq Y$ and $(x,x') \in \partial_\ie f^{-1}(A)$, we have a path of length $\le S$ between $f(x) \not\in A$ and $f(x') \in A$, whence $d(f(x), \partial_\v A), d(f(x'), \partial_\v A) \le S$, i.e., $f(\partial_\v f^{-1}(A)) \subseteq \Ball_S(\partial_\v A)$, whence $\diam(f(\partial_\v f^{-1}(A))) \le \diam(\partial_\v A) + 2S$.
Since $f$ is a coarse embedding, this is enough.

\cref{thm:ends-coarse:bddist}
For $x \in f^{-1}(A) \setminus g^{-1}(A)$, we have a path of length $\le d(f, g)$ between $f(x) \in A$ and $g(x) \not\in A$, whence $d(f(x), \partial_\v A), d(g(x), \partial_\v A) \le d(f, g)$.
Swapping $f, g$, the same holds for $x \in g^{-1}(A) \setminus f^{-1}(A)$.
Thus $f(f^{-1}(A) \triangle g^{-1}(A)) \subseteq \Ball_{d(f,g)}(\partial_\v A)$, which similarly to \cref{thm:ends-coarse:induce} implies the claim.

\cref{thm:ends-coarse:coarseq}
follows by taking a quasi-inverse $g$, so that $d(1_X, g \circ f), d(1_Y, f \circ g) < \infty$.
\end{proof}

\subsection{CBERs and graphings}
\label{sec:cber}

We briefly review here the main definitions we need from descriptive set theory and Borel combinatorics.
For detailed background, see \cite{Kcber}, \cite{KMtopics}.

A \defn{countable Borel equivalence relation (CBER)} $E$ on a standard Borel space $X$ is a Borel equivalence relation $E \subseteq X^2$ with countable equivalence classes, denoted $[x]_E \subseteq X$ for $x \in X$.
More generally, $[A]_E \subseteq X$ denotes the $E$-saturation of $A \subseteq X$.

Every locally countable Borel graph $G \subseteq X^2$ generates a CBER $\#E_G$ whose classes are the $G$-connected components.
Given a CBER $E \subseteq X^2$, we call a Borel graph $G$ with $\#E_G = E$ a \defn{(Borel) graphing} of $E$; we will only ever consider Borel graphings in this paper, and henceforth drop the prefix ``Borel''.
If $\clubsuit$ is a class of connected graphs, then ``$\clubsuit$ing'' will mean a graphing each of whose components is in $\clubsuit$.
In particular, a \defn{treeing} is a graphing which is a forest, i.e., each of whose components is a tree; if a treeing of $E$ exists, then $E$ is called \defn{treeable}.

A CBER $E \subseteq X^2$ is \defn{hyperfinite} if it is a countable increasing union $E = \bigcupup_n F_n$ of \defn{finite Borel equivalence relations (FBERs)} $F_n \subseteq X^2$, meaning with finite equivalence classes.
Such a sequence $(F_n)_n$ is called a \defn{witness to hyperfiniteness} of $E$.

A \defn{Borel homomorphism} $f : (X,E) -> (Y,F)$ between two CBERs is a Borel map $f : X -> Y$ which descends to a map between the quotients $X/E -> Y/F$.

If the descended map is injective (respectively bijective), then $f$ is a \defn{Borel (bi)reduction}; if such $f$ exists, we say $E$ is \defn{Borel (bi)reducible} to $F$.

A CBER is \defn{smooth} if it admits a Borel reduction to equality on a standard Borel space, or equivalently, it admits a Borel selection of a single point in each equivalence class.

A homomorphism between equivalence relations $f : (X,E) -> (Y,F)$ is \defn{class-bijective} if for each $E$-class $C \in X/E$, $f$ restricts to a bijection $C \cong [f(C)]_F$.

A Borel class-bijective homomorphism between CBERs $f : (Y,F) -> (X,E)$ is essentially the same thing as a \defn{Borel action} of the equivalence relation (or groupoid) $E$ on the \defn{Borel bundle} $f : Y -> X$, where each $(x,x') \in E$ acts via the bijection between fibers
\begin{align*}
f^{-1}(x) &--> f^{-1}(x') \\
y &|--> (x,x') \cdot y := \text{the unique $y' \in [y]_F \cap f^{-1}(x')$}.
\end{align*}
Conversely, we can recover $F$ from this action as its orbit equivalence relation.

\section{Treeing median graphs with finite hyperplanes}
\label{sec:cyclecut}

\begin{theorem}
\label{thm:med-subtree}
Let $(X,G)$ be a countable median graph with finite hyperplanes.
Then we may construct a canonical%
\footnote{In the precise sense of \cref{rmk:canonical}; see \cref{thm:cber-median}.}
subtree $T \subseteq G$.
\end{theorem}
\begin{proof}
Let $\@H_\cvx^*(X) = \bigsqcup_{n \in \#N} \@H_n^*$ be a countable coloring of the nontrivial half-spaces such that each $H, \neg H$ receive the same color, so that $\@H_n := \@H_n^* \cup \{\emptyset, X\} \subseteq \@H_\cvx(X)$ is a subpocset, and each $\@H_n$ consists of pairwise nested half-spaces.
To see that such a coloring exists: by \cref{thm:nonnested}, if $H, K \in \@H_\cvx^*(X)$ are non-nested, then in particular, their vertex boundaries obey $\partial_\v H \cap \partial_\v K \ne \emptyset$; so it suffices to take a countable coloring of the intersection graph on these boundaries.

Now let $\@K_n^* := \bigcup_{m \ge n} \@H_m^*$ and $\@K_n := \@K_n^* \cup \{\emptyset, X\}$ for each $n \in \#N$, and consider the quotient median graphs $X/\@K_n$ of $\@K_n$-blocks for each $n$ (in the sense of \cref{thm:median-subpoc}).
Each $(x,y) \in G$ lies in a $\@K_n$-block for sufficiently large $n$, namely for $n$ such that $\cone_x y \in \@H_{n-1}^*$ (and all larger $n$), since $\cone_x y$ is the unique half-space with $x \not\in \cone_x y \ni y$.
Thus the increasing union of the $\@K_n$-block equivalence relations is $X^2$.

We will inductively construct subforests $T_0 \subseteq T_1 \subseteq T_2 \subseteq \dotsb \subseteq G$ such that the connected components of $T_n$ are the $\@K_n$-blocks.
The $\@K_0 = \@H_\cvx(X)$-blocks are singletons (by \cref{thm:half-sep}); thus put $T_0 := \emptyset$.
Suppose we are given a subforest $T_n$ whose components are precisely the $\@K_n$-blocks.
Each $\@K_{n+1}$-block $Y \in X/\@K_{n+1}$ is not separated by any half-spaces in $\@H_m$ for $m > n$, but is separated by $\@H_n$ into the $\@K_n$-blocks contained in $Y$.
Since $\@H_n$ is nested, $Y/\@H_n$ is a tree (\cref{thm:nested-tree}).
For each $Y \in X/\@K_{n+1}$ and each two $G$-adjacent $A, B \in Y/\@H_n$, there are only finitely many edges between $A, B$ (since they are contained in the hyperplane separating $A, B$); we may thus choose one and add it to $T_{n+1}$.
Then $T_{n+1}$ is acyclic, being a forest of blocks which are trees (the $T_n$-components) with a single edge between every pair of adjacent blocks.
Put $T := \bigcupup_n T_n$.
\end{proof}

\begin{example}
The tree $T$ constructed above might not preserve the ends of $G$, in that the map
\begin{align*}
\^\iota : \^X^T --> \^X^G
\end{align*}
induced by the inclusion $\iota : (X,T) -> (X,G)$ might not be injective (it is always surjective since $X$ is dense in $\^X^G$).
In the following bounded degree one-ended median graph (thin black lines), the hyperplanes (bold crossing lines) have been countably colored (from beneath to above), and the edges between adjacent $A, B \in Y/\@H_n$ in the above construction have been chosen, in such a way that the resulting $T$ (thick highlighted edges) has two ends:
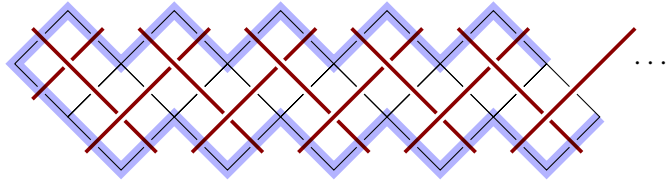
\begin{center}
\begin{tikzpicture}
\begin{scope}[rotate=-45]
\foreach \i in {0,...,4} {
    \draw[graph edge] (\i,\i) -- ++(1,0) -- ++(0,1) -- ++(-1,0) -- ++(0,-1)
        ++(1,0) -- ++(1,0) -- ++(0,1) -- ++(-1,0) -- ++(0,-1);
}
\node at (6,6) {$\dotsb$};
\draw[graph cut cross] (.5,0) -- (.5,1);
\foreach \i in {0,...,4} {
    \draw[graph cut cross] (\i,{\i+.5}) -- ({\i+2},{\i+.5});
    \draw[graph cut cross] ({\i+1.5},\i) -- ({\i+1.5},{\i+2});
}
\draw[graph tree]
    (5,5)
    foreach \i in {4,...,0} {-- (\i,{\i+1}) -- (\i,\i)}
    -- (1,0)
    foreach \i in {0,...,4} {-- ({\i+2},\i) -- ({\i+2},{\i+1})};
\end{scope}
\end{tikzpicture}
\end{center}
\end{example}

In spite of this example, we may tweak the above construction so that

\begin{proposition}
\label{thm:med-subtree-qi}
If $G$ has bounded degree, and hyperplanes in $G$ also have bounded diameters, then the tree $T$ in \cref{thm:med-subtree} may be constructed so that the inclusion $\iota : T -> G$ is a quasi-isometry.
\end{proposition}
In particular, by \cref{thm:ends-coarse}, $\^\iota$ must then be a homeomorphism.
\begin{proof}
We follow the notation and setup from the proof of \cref{thm:med-subtree}.

If $D \ge 2$ is a degree bound for $G$, and $R$ is a bound on the diameters of hyperplanes, then $N := 3^{D^{2R+1}}$ is a bound on the number of colors needed for the coloring of $\@H^*_\cvx$.
Indeed, given $H \in \@H^*_\cvx$, any $K \in \@H^*_\cvx$ non-nested with it must share a boundary vertex with it by \cref{thm:nonnested}.
Thus $K$ is determined by its inner and outer boundaries, which is a pair of disjoint subsets of $\Ball_R(\partial_\v H)$. Because of our diameter bound, $\Ball_R(\partial_\v H) \subseteq \Ball_{2R}(x)$ for any fixed $x \in \partial_\v H$, so
\begin{equation*}
\abs{\Ball_R(\partial_\v H)} \le 1 + D + D^2 + \dotsb + D^{2R} \le D^{2R+1}.
\end{equation*}
Hence, the non-nestedness graph on half-spaces has degree $\le 3^{D^{2R+1}}-1 = N-1$, whence it has an $N$-coloring.

Now we claim that for each $n \le N$ and $\@K_n$-block $Y \in X/\@K_n$, any $G$-adjacent $x, y \in Y$ have
\begin{align*}
d_{T_n}(x, y) \le 1 + 2R + 4R^2 + \dotsb + (2R)^{n-1}
=: M_n.
\end{align*}
In other words, the inverse of the inclusion $(Y,T_n|Y) -> (Y,G|Y)$ (where $|Y$ denotes restriction to $Y$) is $M_n$-Lipschitz.
This is trivial for $n = 0$.
Suppose it holds for $n$; we show it for $n+1$.
If $x, y$ are in the same $\@K_n$-block, then we are done by the induction hypothesis.
Otherwise, the $\@K_n$-blocks containing $x, y$ are separated by the half-space $H := \cone_x y \in \@H_n$.
Let $(x', y') \in \partial_\ie H$ be the unique edge in $Y$ kept during the construction of $T_{n+1}$ from $T_n$.
Since $x, y, x', y' \in \partial_\v H$,
\begin{align*}
d_G(x, x'), d_G(y, y') &\le R, \\
\intertext{whence by the induction hypothesis,}
d_{T_n}(x, x'), d_{T_n}(y, y') &\le RM_n \\
\shortintertext{and so}
d_{T_{n+1}}(x, y) &\le d_{T_{n+1}}(x, x') + d_{T_{n+1}}(x', y') + d_{T_{n+1}}(y', y) \\
&\le 2RM_n + 1 = M_{n+1}.
\end{align*}
Taking $n = N$ completes the proof, since there is a single $\@K_N$-block.
\end{proof}

We now bring these notions into the Borel context:

\begin{definition}
A \defn{(Borel) median graphing} of a CBER $(X,E)$ is a Borel graphing each component of which is a median graph.
\end{definition}

By implementing the proofs of \cref{thm:med-subtree,thm:med-subtree-qi} in a Borel manner, we have

\begin{theorem}
\label{thm:cber-median}
Let $(X,E)$ be a CBER, $G$ be a median graphing of $E$ with finite hyperplanes.
Then there is a Borel subtreeing $T \subseteq G$ of $E$.
If moreover $G$ has bounded degree and hyperplanes with bounded diameters, then $T$ may be constructed so that the inclusion $T -> G$ is a quasi-isometry.
\end{theorem}
\begin{proof}
Let $\@H_\cvx^*(X)$ be the standard Borel space of nontrivial half-spaces in a component of $G$, where we identify a nontrivial half-space $H$ with the pair $(\partial_\iv H, \partial_\ov H)$ of nonempty sets of vertices.
By \cite[7.3]{KMtopics} and Lusin--Novikov, there is a Borel countable coloring $\@H_\cvx^*(X) = \bigsqcup_{n \in \#N} \@H_n^*$ such that each $H, \neg H$ receive the same color and any two distinct non-complementary half-spaces with intersecting vertex boundary receive distinct colors.
Define inductively Borel subforests $\emptyset =: T_0 \subseteq T_1 \subseteq \dotsb \subseteq G$ by taking $T_{n+1}$ to be $T_n$ together with a single edge on the hyperplane separating any two $G$-adjacent $T_n$-components in the same $(\@K_n = \bigcup_{m \ge n} \@H_m)$-block, chosen using Lusin--Novikov, and put $T := \bigcup_n T_n$; this works by the proof of \cref{thm:med-subtree}.

In the bounded-degree case, use \cite[4.6]{KSTchrom} to choose a finite Borel coloring $\@H_\cvx^*(X) = \bigsqcup_{n < N} \@H_n^*$ to begin with, where $N$ is as in the proof of \cref{thm:med-subtree-qi}.
\end{proof}

\section{Proper wallspaces}
\label{sec:walls}

In the rest of the paper, we apply the results of the previous section to show treeability of CBERs equipped with various kinds of geometric structures from which a median graph may be constructed in a canonical manner (in the precise sense of \cref{rmk:canonical}).
We begin in this section with a fairly general kind of such structure, before specializing in the following section to ``tree-like'' graphs.

Recall \cref{def:walls} of \emph{wallspace}: a set $X$ equipped with a subpocset $\@H = \@H(X) \subseteq 2^X$, called the \emph{walling} of the wallspace, which is \emph{finitely separating}, meaning
\begin{enumerate}[label=(\roman*)]
\item \label{ax:walls-finsep}
For any $x, y \in X$, there are only finitely many $H \in \@H(X)$ with $x \in H \not\ni y$.
\end{enumerate}
By \cref{thm:subpoc-isolated}, this means that $\@H(X) \subseteq 2^X$ is closed and has nontrivial elements isolated.

\begin{definition}
\label{def:walls-proper}
We call a wallspace $X$ \defn{proper} if it additionally satisfies:
\begin{enumerate}[resume*]
\item \label{ax:walls-finblock}
For any $x \in X$, there are only finitely many $y \in X$ with $\forall H \in \@H(X)\, (x \in H \iff y \in H)$, i.e., each $\@H$-block is finite.
\item \label{ax:walls-finnest}
For any $H \in \@H(X)$, there are only finitely many $K \in \@H(X)$ non-nested with $H$.
\item \label{ax:walls-fincover}
For any $H \in \@H^*(X)$, there are only finitely many successors $H \subsetneq K \in \@H^*(X)$.
\end{enumerate}
By \cref{thm:med-hyp-locfin}, the last two conditions mean that the dual median graph $\@U^\circ(\@H(X))$ is locally finite and has finite hyperplanes.
Condition \cref{ax:walls-finblock} means that $x |-> \^x : X -> \@U^\circ(\@H(X))$ is finite-to-one.
\end{definition}

\begin{definition}
\label{def:walling}
A \defn{Borel walling} of a CBER $(X,E)$ is a ``Borel assignment'' of a walling $\@H(C) \subseteq 2^C$, i.e., subpocset satisfying \cref{def:walls}, for each $E$-class $C \in X/E$.

To make this precise, we may first construct the Borel bundle (see \cref{sec:cber}) of powersets
\begin{align*}
2^E_X := \bigsqcup_{x \in X} 2^{[x]_E} := \{(x,A) \mid x \in X \AND A \subseteq [x]_E\}
\end{align*}
equipped with the first projection $p : \bigsqcup_{x \in X} 2^{[x]_E} -> X$, and the Borel structure generated by declaring, for any partial Borel map $f : X \rightharpoonup X$ with graph contained in $E$, the set
\begin{align*}
\{(x,A) \mid x \in \dom(f) \AND f(x) \in A\} \subseteq 2^E_X
\end{align*}
to be Borel.
Equivalently, given a Borel family of bijections $(e_x : \#N \cong [x]_E)_{x \in X}$ via the Lusin--Novikov uniformization theorem, we may identify $2^E_X$ with $X \times 2^{\#N}$ with the product Borel structure via the isomorphism $(x,A) |-> (x,e_x^{-1}(A))$.
Note that $E$ acts on this bundle $p : 2^E_X -> X$ by transporting between fibers $p^{-1}(x), p^{-1}(y)$ over $E$-related points $(x, y) \in E$:
\begin{equation*}
(x,y) \cdot (x,A) := (y,A).
\end{equation*}
Now the condition for calling the walling Borel is that the set
\begin{align*}
\@H_X(E) := \bigsqcup_{x \in X} \@H([x]_E) = \{(x,H) \mid H \in \@H([x]_E)\} \subseteq 2^E_X
\end{align*}
is Borel.
Note that it is then an $E$-invariant Borel subset of $2^E_X$, to which the action $E \actson 2^E_X$ restricts; the orbit equivalence relation of this action is then a CBER on the space $\@H_X(E)$.

A \defn{proper Borel walling} is a Borel walling $\@H_X(E)$ such that for each $C \in X/E$, $\@H(C)$ is a proper walling on $C$.
We call a CBER $E$ \defn{properly wallable} if it admits a proper Borel walling.
\end{definition}

\begin{lemma}
\label{thm:walling-smooth}
Let $\@H = \@H_X(E) \subseteq 2^E_X$ be a proper Borel walling of a CBER $(X,E)$.
Then the orbit equivalence relation of the action $E \actson \@H_X^*(E)$ is smooth.
\end{lemma}
\begin{proof}
We have a Borel complete section $S \subseteq \@H_X^*(E)$ selecting finitely many points from each orbit, namely all those $(x,H) \in \@H_X^*(E)$, i.e., $H \in \@H^*([x]_E)$, such that in the median graph $\@U^\circ(\@H([x]_E))$, the corresponding principal orientation $\^x \in \^H$ is at minimal distance from the set of all $\^y \in \neg \^H$; there exist such $\^x, \^y$ since $H \ne \emptyset, X$, the set of nearest $\^x \in \@U^\circ(\@H([x]_E))$ is finite for each $H$ by the axioms of \cref{def:walls-proper}, and the entire set $S$ is easily seen to be Borel using \cref{thm:half-dist}.
\end{proof}

\begin{remark}
The set $S$ in the proof above can be seen as selecting the points $x$ ``approximately on the inner boundary'' of each half-space $H$.
The ``approximate'' is needed because the genuine inner boundary in the median graph $\@U^\circ(\@H([x]_E))$ may not contain any principal orientations, as in \cref{fig:wall-boundary-nonunique}, which also shows that this ``approximate boundary'' $S$ may not uniquely identify $H$.

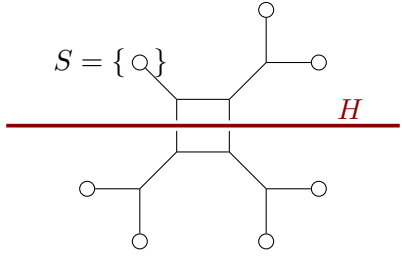
\begin{figure}[htb]
\centering
\begin{tikzpicture}[every edge/.style={graph edge}, scale=.7, every node/.style={inner sep=2pt}]
\path
    (0,0) coordinate(a)
        ++(-.7,-.7) coordinate(a0) edge (a)
            +(-1,0) node[draw,circle](a00){} edge (a0)
            +(0,-1) node[draw,circle](a01){} edge (a0)
    (1,0) coordinate(b) edge (a)
        ++(.7,-.7) coordinate(b0) edge (b)
            +(1,0) node[draw,circle](b00){} edge (b0)
            +(0,-1) node[draw,circle](b01){} edge (b0)
    (1,1) coordinate(c) edge (b)
        ++(.7,.7) coordinate(c0) edge (c)
            +(1,0) node[draw,circle](c00){} edge (c0)
            +(0,1) node[draw,circle](c01){} edge (c0)
    (0,1) coordinate(d) edge (a) edge (c)
        ++(-.7,.7) node[draw,circle,label={left:$S=\{$},label=right:$\}$](d0){} edge (d);
\draw[graph cut cross] (-3,.5) -- node[pos=.9,above]{$H$} (4,.5);
\end{tikzpicture}
\caption{A dual median graph with a half-space $H$ with no principal orientations (circled vertices) on its boundary, and such that the set $S$ of principal orientations in $H$ closest to its boundary do not uniquely identify it (the left side of the vertical hyperplane down the middle would have the same $S$).}
\label{fig:wall-boundary-nonunique}
\end{figure}

While such uniqueness is not required for smoothness of $E \actson \@H_X^*(E)$, we would like to point out that it \emph{is} possible to define a more involved (but still canonical) finite ``approximate boundary'' of $H \in \@H^*(X)$ in a proper wallspace $X$ that uniquely identifies $H$.
Namely, we may take the union of all minimal-diameter subsets of $X$ that generate $\partial_\v H$ under the median operation $\ang{\blank,\blank,\blank}$, and then remember the restrictions of $H, \neg H$ to this finite set.
\end{remark}

\begin{remark}
\label{rmk:walling-concrete}
The significance of \Cref{thm:walling-smooth} is that we may construct a standard Borel space $\@H_X^*(E)/E$ of all nontrivial half-spaces in all equivalence classes, rather than merely a bundle of such spaces varying over a basepoint.
This space is key in the main \cref{thm:cber-wallable-treeable} of this section.

However, in concrete examples of wallings where there is a natural notion of finite ``boundary'', such as cuts in graphs as in \cref{sec:cuts}, it is easier to bypass the bundle $\@H_X^*(E)$ and directly construct the standard Borel space of the half-spaces in the given walling by representing each half-space $H$ via its finite ``boundary''.
\end{remark}

\begin{theorem}
\label{thm:cber-wallable-treeable}
A CBER $(X,E)$ is properly wallable iff it is treeable.
\end{theorem}
\begin{proof}
Given a proper walling $\@H_X(E) \subseteq 2^E_X$, let $\@U^\circ_X(\@H_X(E))/E$ be the Borel (locally finite) median graph with finite hyperplanes given by the disjoint union of $\@U^\circ(\@H(C))$ for each $C \in X/E$, where each clopen orientation $U \in \@U^\circ(\@H(C))$ is represented as its finite set of minimal elements in the standard Borel space of half-spaces $\@H_X^*(E)/E$.
(Equivalently, we first apply $\@U^\circ$ fiberwise to the bundle of pocsets $\@H_X(E)$, and then quotient by the $E$-action, hence the notation.)
By \cref{thm:cber-median}, the quotient median graph on $\@U^\circ_X(\@H_X(E))/E$ is treeable, and Borel bireducible with $E$ via $X \ni x |-> \^x \in \@U^\circ(\@H([x]_E))$, whence $E$ is also treeable by \cite[3.3]{JKLcber}.

Conversely, if $E$ is treeable, then it is locally finite treeable by \cite[3.12]{JKLcber}; the half-spaces of such a treeing yield a proper walling.
\end{proof}

\begin{remark}
\label{rmk:walls-proper-borel}
Given a CBER $(X,E)$ and a Borel assignment $C |-> \@H(C) \subseteq 2^C$ as in \cref{def:walling}, the union of those $C \in X/E$ in which $\@H(C)$ forms a proper walling is Borel; this easily follows from \cref{def:walls-proper}.
Thus, if we have countably many such Borel assignments $(\@H_i)_{i \in \#N}$, such that each $C \in X/E$ admits at least one $\@H_i(C)$ which is a proper walling, then we may combine these $\@H_i$ into a single $\@H$ which is a proper walling, whence $E$ is again treeable.
\end{remark}

\section{Graphs with dense families of cuts}
\label{sec:cuts}

In this section, we provide a general method for identifying proper wallings consisting of \emph{cuts} in graphs (i.e., ways of splitting the vertex set into two pieces) obeying various graph-theoretic properties.
We then apply this method to two specific classes of graphs, namely, quasi-trees and bounded tree-width graphs, to deduce that such graphings are treeable.

\subsection{Wallings of cuts}

Let $(X,G)$ be a connected locally finite graph.
As in \cref{sec:ends}, $\@H_\fin(X) = \@H_\fin^G(X) \subseteq 2^X$ denotes the Boolean algebra of subsets with finite $G$-boundary, whose Stone space $\^X$ is the end compactification of $X$.

\begin{remark}
\label{rmk:graph-walls}
Given a graph $(X,G)$ as above and a walling $\@H \subseteq 2^X$, to say that such $\@H$ is contained in $\@H_\fin(X)$ (which is \emph{not} usually a walling, due to not being closed in $2^X$) means precisely that the principal orientations map $X -> \@U^\circ(\@H)$ extends continuously to the end compactifications:
\begin{equation*}
\begin{tikzcd}
X \dar[hook] \rar & \@U^\circ(\@H) \dar[hook] \\
\^X \rar[dashed] & \^{\@U^\circ(\@H)}
\end{tikzcd}
\end{equation*}
Indeed, such an extension exists iff preimage under $X -> \@U^\circ(\@H)$ preserves $\@H_\fin$, which by \cref{thm:med-fin-cvx} means precisely that the preimage $H$ of each half-space $\^H \subseteq \@U^\circ(\@H)$ is in $\@H_\fin(X)$.

Note also that if $\@U^\circ(\@H)$ has finite hyperplanes, so that $\^{\@U^\circ(\@H)} \cong \@U(\@H)$ by \cref{thm:med-ends}, the induced map above becomes simply
\begin{align*}
2^{\@H_\fin} \supseteq \^X &--> \@U(\@H) \subseteq 2^{\@H} \\
U &|--> \{H \in \@H \mid H \in U\}.
\end{align*}
By \cref{thm:walls-ends-dense}, this map is surjective on ends.
\end{remark}

\begin{lemma}
\label{thm:graph-walling-fin}
For a subpocset $\@H \subseteq \@H_\fin(X)$, the following are equivalent:
\begin{enumerate}[label=(\roman*)]
\item \label{thm:graph-walling-fin:walling}
$\@H$ is a walling, i.e., any $x, y \in X$ are separated by only finitely many $H \in \@H$.
\item \label{thm:graph-walling-fin:fin}
Each $x \in X$ is on the boundary of only finitely many $H \in \@H$.
\end{enumerate}
\end{lemma}
\begin{proof}
If \cref{thm:graph-walling-fin:walling} holds, then each $x \in X$ is separated from each of its finitely many neighbors by only finitely many $H \in \@H$, proving \cref{thm:graph-walling-fin:fin}.
Conversely, if \cref{thm:graph-walling-fin:fin} holds, then for $x, y \in X$, pick any path between them; $H \in \@H$ separating $x, y$ must separate some edge along this path, and there are only finitely many such $H$ for each edge, proving \cref{thm:graph-walling-fin:walling}.
\end{proof}

This gives a graph-theoretic reformulation of axiom \ref{def:walls-proper}\cref{ax:walls-finsep} of walling.
To reformulate the remaining axioms of proper walling, we need to impose connectedness:

\begin{definition}
Let $\@H_\conn(X) = \@H^G_\conn(X) \subseteq 2^X$ denote the subpocset of $A \subseteq X$ such that $A, \neg A$ are connected or empty, and as usual (\cref{conv:H}), $\@H_\conn^* := \@H_\conn \setminus \{\emptyset, X\}$.
\end{definition}

\begin{lemma}
\label{thm:cut-limit}
No $A \in \@H_\fin^*(X)$ is a limit point (in $2^X$) of $\@H_\conn(X)$.
\end{lemma}
\begin{proof}
$A$ has a clopen neighborhood of all $B \subseteq X$ containing $\partial_\iv A$ and disjoint from $\partial_\ov A$, which is disjoint from $\@H_\conn(X) \setminus \{A\}$.
\end{proof}

\begin{corollary}
$\@H_\fin(X) \cap \smash{\-{\@H_\conn(X)}} = \@H_\fin(X) \cap \@H_\conn(X)$, with nontrivial elements isolated.
\qed
\end{corollary}

\begin{corollary}
\label{thm:cut-walling}
A subpocset $\@H \subseteq \@H_\fin \cap \@H_\conn$ is a walling iff it is closed in $2^X$.
\qed
\end{corollary}

We have two main examples of families $\@H$ for which the above conditions are obvious:

\begin{example}
\label{ex:cut-diam}
For any $R \ge 0$, $\@H_{\diamle R} := \{H \subseteq X \mid \diam(\partial_\v H) \le R\}$ is clearly closed in $2^X$, and also clearly finitely separating.
Thus, $\@H_{\diamle R} \cap \@H_\conn$ is a walling.
\end{example}

\begin{example}
\label{ex:cut-card}
For any $N \ge 0$, $\@H_{\cardle{N}} := \{H \subseteq X \mid \min(\abs{\partial_\iv H}, \abs{\partial_\ov H}) \le N\}$ is closed in $2^X$.
Thus, $\@H_{\cardle{N}} \cap \@H_\conn$ is a walling.
(However, $\@H_{\cardle{N}}$ may not be, e.g., in a $\#Z$-line, since $\#N \in \@H_{\cardle2}$ is not an isolated point, being the limit of $\#N \cup \{-1\}, \#N \cup \{-2\}, \dotsc \in \@H_{\cardle2}$.)
\end{example}

We now turn to proper wallings, i.e., axioms \ref{def:walls-proper}\cref{ax:walls-finblock}, \cref{ax:walls-finnest}, and \cref{ax:walls-fincover}; among these, \cref{ax:walls-finnest} is automatic:

\begin{lemma}
\label{thm:cut-nonnested}
For a walling $\@H \subseteq \@H_\fin \cap \@H_\conn$, each $H \in \@H_\fin$ is non-nested with only finitely many $K \in \@H$.
Thus by \cref{thm:med-hyp-fin}, the median graph $\@U^\circ(\@H)$ has finite hyperplanes.
\end{lemma}
\begin{proof}
Let $H \in \@H_\fin$ be non-nested with $K \in \@H$.
There is a path in $K$ between $H \cap K, \neg H \cap K$, whence $\partial_\v H \cap K \ne \emptyset$; similarly $\partial_\v H \cap \neg K \ne \emptyset$.
Now any path between $\partial_\v H \cap K, \partial_\v H \cap \neg K$ must contain a vertex on $\partial_\v K$.
Thus, fixing for each $x, y \in \partial_\v H$ a path $p_{xy}$ between them, we have that any $K$ non-nested with $H$ must contain some $z$ on some $p_{xy}$ on its boundary $\partial_\v K$.
So given $H$, there are finitely many $x, y \in \partial_\v H$, for each of which there are finitely many $z$ on $p_{xy}$, for each of which there are finitely many $K \in \@H$ with $z \in \partial_\v K$ by \cref{thm:graph-walling-fin}.
\end{proof}

The following proposition characterizes the remaining axioms (\ref{def:walls-proper}\cref{ax:walls-finblock} and \cref{ax:walls-fincover}) of proper walling for cuts in graphs. For this, we introduce a density notion for cuts.

\begin{definition}
We call a family $\@H \subseteq \@H_\fin$ \defn{dense towards ends} (of $G$) if $\@H$ contains a neighborhood basis for every end.
\end{definition}

\begin{remark}
For $\@H$ as above, being dense towards ends is strictly weaker than ``$\@H$ forms a basis for the topology of $\^X$'' (since it needs only separate vertices of $X$ into finite classes), but is strictly stronger than ``$\@H$ restricts to a basis for the topology of $\^X \setminus X$'' (e.g., in a one-ended graph, any single cofinite $H \subseteq X$ yields a basis for the single end in $\^X \setminus X$, but $\@H$ that is dense towards ends must be infinite).
\end{remark}

\begin{proposition}
\label{thm:cut-proper}
For a walling $\@H \subseteq \@H_\fin \cap \@H_\conn$, the following are equivalent:
\begin{enumerate}[label=(\roman*)]
\item \label{thm:cut-proper:proper}
$\@H$ is a proper walling, i.e., obeys \labelcref{def:walls-proper}\cref{ax:walls-finblock} and \labelcref{def:walls-proper}\cref{ax:walls-fincover} (and also \labelcref{def:walls-proper}\cref{ax:walls-finnest} by \cref{thm:cut-nonnested}).
\item \label{thm:cut-proper:dense}
$\@H$ is dense towards ends of $G$.
\item \label{thm:cut-proper:inj}
The induced map $\^X -> \^{\@U^\circ(\@H)} \cong \@U(\@H)$ from \cref{rmk:graph-walls} takes ends to ends, and is injective on ends (hence bijective on ends by \cref{thm:walls-ends-dense}).
\end{enumerate}
\end{proposition}
\begin{proof}
\cref{thm:cut-proper:dense}$\implies$\cref{thm:cut-proper:inj}:
An end cannot map to a vertex, since it is non-isolated, hence does not have a minimal neighborhood in $\@H$ by \cref{thm:cut-proper:dense}.
And two ends cannot map to the same end, since they are separated by some $H \in \@H$ by \cref{thm:cut-proper:dense}.

\cref{thm:cut-proper:inj}$\implies$\cref{thm:cut-proper:dense}:
By \cref{thm:cut-nonnested}, $\@U^\circ(\@H)$ has finite hyperplanes.
Thus by \cref{thm:med-ends-half-basis}, half-spaces form a neighborhood basis for ends in $\^{\@U^\circ(\@H)}$.
Now given an end $U \in \^X \setminus X$ and a clopen neighborhood $\^A \ni U$, by \cref{thm:cut-proper:inj}, the image of $U$ under the induced map is not in the image of $\neg \^A$, and the latter is closed by compactness, hence is disjoint from some half-space $\^H$ containing $U$.

\cref{thm:cut-proper:dense} + \cref{thm:cut-proper:inj}$\implies$\cref{thm:cut-proper:proper}:
By \cref{thm:cut-nonnested}, it remains to check \ref{def:walls-proper}\cref{ax:walls-finblock} that each $\@H$-block is finite and \ref{def:walls-proper}\cref{ax:walls-fincover} each $H \in \@H^*$ has only finitely many successors.
The former follows from the fact that no end in $\^X$ maps to a vertex in $\@U^\circ(\@H)$.
For the latter, let $H \in \@H^*$.
For any end $U \in \neg \^H$, since $\neg \^H$ is not a minimal neighborhood of $U$, there is some $H \subsetneq K \in \@H^*$ with $U \in \neg \^K$; by taking $K$ minimal (which is possible since $H$ is isolated), we get $U \in \neg \^K$ for some successor $K$ of $H$.
Thus by compactness of $\neg \^H \setminus X$, there are finitely many successors $H \subsetneq K_1, \dotsc, K_n \in \@H^*$ such that every end in $\neg \^H$ is in some $\neg \^K_i$.
It follows by compactness that $K_1 \cap \dotsb \cap K_n \setminus H$ must be finite.
Now let $H \subsetneq K \in \@H^*$ be another successor not equal to any $K_i$; thus for each $i$, we have $K \setminus K_i, K_i \setminus K, K \cap K_i \ne \emptyset$ (the last because $H \subseteq K \cap K_i$).
If $\neg K \cap \neg K_i \ne \emptyset$ for some $i$, then $K, K_i$ are non-nested; by \cref{thm:cut-nonnested}, there are only finitely many possibilities for $K$ for each $i$.
Otherwise, we have $\neg K \subseteq K_1 \cap \dotsb \cap K_n \setminus H$, so there are again only finitely many such $K$.

\cref{thm:cut-proper:proper}$\implies$\cref{thm:cut-proper:inj}:
Since each $\@H$-block is finite by \ref{def:walls-proper}\cref{ax:walls-finblock}, no end in $\^X$ maps to a vertex in $\@U^\circ(\@H)$.
Suppose two distinct ends $U, V \in \^X$ map to the same end $W \in \^{\@U^\circ(\@H)}$.
Since half-spaces $\^H \ni W$ form a neighborhood basis for $W$ by \cref{thm:med-ends-half-basis}, it follows that the collection of $H \in \@H$ containing both $U, V$ (as opposed to neither) is an ultrafilter basis.
But then the intersection of all such $H$ is a set of vertices in $X$ not separated by any element of $\@H$, which is infinite because e.g., for any $K \in \@H$ separating $U$ from $V$, each such $H$ must meet $\partial_\v K$ which is a finite set, and we can find infinitely many such $K$ with pairwise disjoint boundaries.
So \ref{def:walls-proper}\cref{ax:walls-finblock} fails.
\end{proof}

\Cref{thm:cut-proper} is formulated only for $\@H \subseteq \@H_\conn$, which is needed in order for \cref{thm:cut-walling,thm:cut-nonnested} to go through (see \cref{ex:cut-card}).
However, a general $\@H \subseteq \@H_\fin$ dense towards ends can be converted into one contained in $\@H_\conn$:

\begin{lemma}
\label{thm:cut-flipflip}
Let $\@H \subseteq \@H_\fin$ be dense towards ends.
Then there is $\@H' \subseteq \@H_\fin \cap \@H_\conn$ which is also dense towards ends, and such that every $H' \in \@H'$ has $\partial_\ie H' \subseteq \partial_\ie H$ for some $H \in \@H$.
\end{lemma}

\begin{figure}[htb]
\centering
\begin{tikzpicture}
\coordinate (G);
\coordinate (GO) at ($(G) + (8,0)$);
\coordinate (GL) at ($(GO) + (-7,0)$);
\coordinate (GL1) at ($(GL) + (0,1.5)$);
\coordinate (GL2) at ($(GL) + (0,-1)$);
\coordinate (GR) at ($(GO) + (5,0)$);
\coordinate (GR1) at ($(GR) + (0,1.5)$);
\coordinate (GR2) at ($(GR) + (0,-1.5)$);
\def\toparc{($(GR1) + (0,.5)$)
    to ($(GL1) + (0,.5)$)}
\def\leftarc{($(GL1) + (0,-.5)$)
    to[out=0,in=90,looseness=.2] ($(GO) + (-3,1)$)
    to[out=-90,in=0,looseness=.5] ($(GL2) + (0,.5)$)}
\def\botarc{($(GL2) + (0,-.5)$)
    to[out=0,in=180,looseness=1] ($(GO) + (0,1)$)
    to[out=0,in=0,looseness=10] ($(GO) + (0,.5)$)
    to[out=180,in=180,looseness=3] ($(GO) + (0,-2)$)
    to ($(GR2) + (0,-.5)$)}
\def\rightarc{($(GR2) + (0,.5)$)
    to ($(GO) + (0,-1)$)
    to[out=180,in=180,looseness=8] ($(GO) + (0,-.5)$)
    to[out=0,in=180,looseness=1] ($(GR1) + (0,-.5)$)}
\fill[graph interior]
    \toparc -- \leftarc -- \botarc -- \rightarc -- cycle;
\draw \toparc \leftarc \botarc \rightarc;
\node at (GR1) {$\quad \longrightarrow U$};
\draw[graph cut]
    ($(GL1) + (2,1.25)$) rectangle ($(GL2) + (2,-1.25)$);
\node[graph cut,draw=none,below right] at ($(GL1) + (2,1.25)$) {$A$};
\node[graph cut,draw=none,below left] at ($(GL1) + (2,1.25)$) {$\neg A$};
\pattern[pattern=north east lines, pattern color=black, xshift=1cm, yscale=-1, yshift=-3.15cm]
    svg "M 51.023438 33.04248 L 51.023438 58.510254 C 54.82942 58.326809 58.559971 58.135647 62.362793 57.990234 L 62.362793 49.731445 C 70.989469 49.794886 84.537068 50.084814 97.136719 51.492188 C 105.01446 52.372127 112.42856 53.708265 117.62842 55.574707 C 120.22834 56.507928 122.25787 57.587227 123.47607 58.628906 C 124.69429 59.670585 125.07275 60.464575 125.07275 61.390137 C 125.07275 63.962138 122.28551 67.771407 116.97217 71.707031 C 111.65881 75.642656 104.21324 79.65903 96.298828 83.427246 C 84.227854 89.174485 71.215464 94.275807 62.362793 98.308594 L 62.362793 92.438965 C 58.599646 93.529844 54.817564 94.549142 51.023438 95.503418 L 51.023438 122.39648 C 54.88518 120.88493 58.661813 119.24642 62.362793 117.49805 L 62.362793 103.83984 C 70.650594 99.951648 85.128247 94.264683 98.441895 87.925781 C 106.49717 84.090498 114.14663 79.999418 119.93701 75.710449 C 125.7274 71.421482 130.05322 66.888431 130.05322 61.390137 C 130.05322 58.792267 128.65668 56.502719 126.71484 54.842285 C 124.77302 53.181851 122.25609 51.940676 119.31299 50.884277 C 113.4268 48.771478 105.76676 47.444765 97.688965 46.54248 C 84.754006 45.097654 71.041613 44.810454 62.362793 44.749512 L 62.362793 33.04248 L 51.023438 33.04248 z";
\node[right] at ($(GO) + (-2.5,1)$) {$B$};
\draw[graph cut]
    ($(GO) + (0,2.75)$) rectangle ($(GO) + (0,-2.25)$);
\node[graph cut,draw=none,below right] at ($(GO) + (0,2.75)$) {$H$};
\node[graph cut,draw=none,below left] at ($(GO) + (0,2.75)$) {$\neg H$};
\begin{scope}
\clip ($(GO) + (0,-.75)$) rectangle ($(GR1) + (0,1)$);
\pattern[pattern=north west lines, pattern color=black]
    \toparc -- \leftarc -- \botarc -- \rightarc -- cycle;
\end{scope}
\node[right,xshift=1em] at (GO |- GR1) {$C$};
\end{tikzpicture}
\caption{Shrinking a neighborhood $A$ of an end $U$ to a connected-coconnected subneighborhood.}
\label{fig:cut-flipflip}
\end{figure}
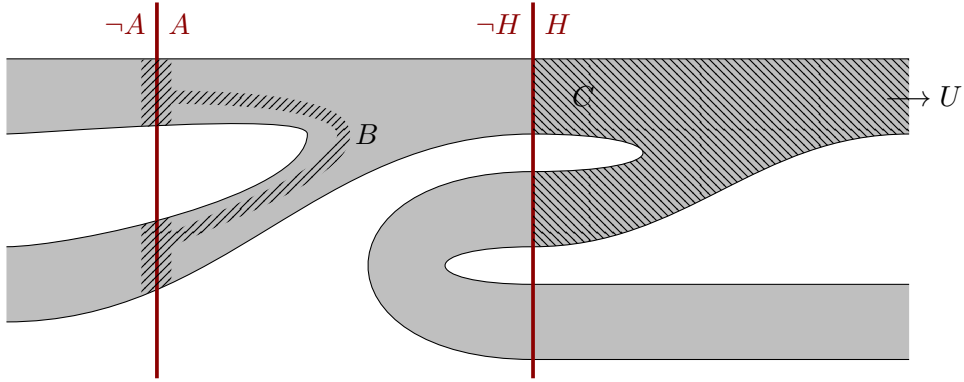

\begin{proof}
Let
\begin{equation*}
\@H' := \{\neg D \mid H \in \@H \AND C \in H/G \AND D \in \neg C/G\}
\end{equation*}
(where $H/G$ denotes the set of $G$-components of $H$).
Then every such $\neg D \in \@H'$ has $\partial_\ie \neg D \subseteq \partial_\ie C \subseteq \partial_\ie H$.
Now fix an end $U \in \^X \setminus X$ and a neighborhood $A \in \@H_\fin$ of $U$; we must find a smaller neighborhood in $\@H'$.
Let $\partial_\v A \subseteq B \subseteq X$ be finite connected.
Then $\neg B$ is a neighborhood of $U$, so there is $\neg B \supseteq H \in \@H$ such that $U \in \^H$.
Let $C \subseteq H$ be the component containing $U$, and $D \subseteq \neg C$ be the component containing $B$.
Then $\neg D \in \@H'$, with $U \in \^C \subseteq \neg \^D$, and $\neg D \subseteq A$ since $\neg D$ is connected, disjoint from $\partial_\v A \subseteq B$, and contains $U \in \^A$.
(See \cref{fig:cut-flipflip}.)
\end{proof}

\begin{corollary}
\label{thm:cut-boundary}
Let $\@H \subseteq \@H_\fin$ be a closed subpocset which is ``downward-closed under boundary inclusion'', i.e., if $H \in \@H$ and $H' \subseteq X$ with $\partial_\ie H' \subseteq \partial_\ie H$, then $H' \in \@H$.
If $\@H$ is dense towards ends, then so is $\@H \cap \@H_\conn$, which is thus a proper walling.
\qed
\end{corollary}

\begin{example}
$\@H_{\diamle R}$ and $\@H_{\cardle{N}}$ are closed subpocsets of $\@H_\fin$ (by \cref{ex:cut-diam,ex:cut-card}) which are by definition clearly downward-closed under boundary inclusion.
Thus if they are dense towards ends, then their intersection with $\@H_\conn$ forms a proper walling.
\end{example}

In the Borel context, this yields

\begin{corollary}[of \cref{thm:cber-wallable-treeable,thm:cut-proper}]
\label{thm:cber-cuts-dense-treeable}
If a CBER $(X,E)$ admits a graphing $G \subseteq E$ with a Borel assignment $C |-> \@H(C) \subseteq \@H_\fin^G(C) \cap \@H_\conn^G(C) \subseteq 2^C$ of a walling which is dense towards ends in each component $C \in X/E$, then $E$ is treeable.
\qed
\end{corollary}

\begin{corollary}[of \cref{thm:cut-boundary,thm:cber-cuts-dense-treeable}]
\label{thm:cber-dense-treeable}
If a CBER $(X,E)$ admits a graphing $G \subseteq E$ with a Borel assignment $C |-> \@H(C) \subseteq \@H_\fin^G(C)$ of a closed subpocset which is downward-closed under boundary inclusion and dense towards ends in each component $C \in X/E$, then $E$ is treeable.
\qed
\end{corollary}

Here ``Borel assignment $C |-> \@H(C) \subseteq 2^C$'' may be interpreted via the bundle in \cref{def:walling}.
However, as mentioned in \cref{rmk:walling-concrete}, since we are dealing with cuts in a graph, we may equivalently represent each nontrivial $H \in \@H^*(C)$ as the pair of finite sets $(\partial_\iv H, \partial_\ov H)$, and hence realize $\@H$ globally as a Borel subset of the standard Borel space of pairs of finite subsets of $E$-classes.

\subsection{Coarse equivalences and quasi-trees}

Recall \cref{def:coarse} of coarse equivalence.

\begin{lemma}
\label{thm:quasitree-coarse}
The class of connected locally finite graphs in which $\@H_{\diamle R}$ is dense towards ends for some $R < \infty$ is invariant under coarse equivalence.
\end{lemma}
\begin{proof}
Let $(X,G), (Y,T)$ be connected locally finite graphs, $f : X -> Y$ be a coarse equivalence with quasi-inverse $g : Y -> X$, and suppose $\@H_{\diamle S}(Y)$ is dense towards ends for $S < \infty$.
By \cref{thm:ends-coarse}, pick $R < \infty$ so that for any $H \in \@H_{\diamle S}(Y)$, we have $f^{-1}(H) \in \@H_{\diamle R}(X)$.
Then for any $U \in \^X \setminus X$ and $A \in \@H_\fin(X)$ with $U \in \^A$, letting $B := \neg \Ball_{d(1_X, g \circ f)}(\neg A)$, we have $f^{-1}(g^{-1}(B)) \subseteq \Ball_{d(1_X, g \circ f)}(B) \subseteq A$, and $A \triangle B, B \triangle f^{-1}(g^{-1}(B))$ are finite, so $U \in {f^{-1}(g^{-1}(B))}^\wedge$, so $\^f(U) \in \^{g^{-1}(B)}$, so there is $g^{-1}(B) \supseteq H \in \@H_{\diamle S}(Y)$ with $\^f(U) \in H$, so $f^{-1}(H) \in \@H_{\diamle R}(X)$ with $U \in \^{f^{-1}(H)}$ and $f^{-1}(H) \subseteq f^{-1}(g^{-1}(B)) \subseteq A$.
\end{proof}

Recall (\cref{sec:intro-treelike}) that a graph $(X,G)$ is a \defn{quasi-tree} if it is quasi-isometric to a tree.

\begin{remark}
\label{rmk:quasitree-coarse}
A graph is a quasi-tree iff it is coarsely equivalent to a tree; this follows easily from the fact that graphs are quasi-geodesic (see \cite[p7]{Groggt2}).
\end{remark}

\begin{corollary}
\label{thm:quasitree-dense}
If $G$ is a locally finite quasi-tree, then $\@H_{\diamle R}$ is dense towards ends of $G$ for some $R < \infty$.
\end{corollary}
\begin{proof}
By \cref{thm:quasitree-coarse} and that $\@H_{\diamle 1}$ is clearly dense towards ends in a tree.
\end{proof}

\begin{corollary}
\label{thm:cber-quasitreeable-treeable}
If a CBER admits a locally finite graphing whose components are quasi-trees, then it is treeable.
\end{corollary}
\begin{proof}
By \cref{thm:quasitree-dense}, \cref{thm:cber-dense-treeable}, and \cref{rmk:walls-proper-borel}.
\end{proof}

\begin{remark}
\Cref{ex:dense-nonquasitree-nonbddtw} shows that the class of graphs in which $\@H_{\diamle R}$ is dense towards ends for some $R$ is strictly bigger than just quasi-trees.
\end{remark}

In fact, the following result gives a more explicit walling in a quasi-tree:

\begin{theorem}[{Krön--Möller \cite[2.8]{KMqi}}]
\label{thm:kron-moller}
A connected graph $(X,G)$ is a quasi-tree iff there is some $R < \infty$ such that for every $x \in X$, $S < \infty$, and component $H \subseteq \neg \Ball_S(x)$, we have $\diam(\partial_\v H) \le R$, i.e., all $H$ of this form, called \defn{radial cuts}, are in $\@H_{\diamle R}(X)$.
\end{theorem}

This easily yields another proof of \cref{thm:quasitree-dense}.
Another consequence is the following quantitative refinement of the above results:

\begin{lemma}
\label{thm:quasitree-cut-adj}
If $G$ is a locally finite quasi-tree, then for all sufficiently large $R < \infty$, there is an $S < \infty$, namely $S := R+1$, such that any pair of sets $H, K \in \@H^*_{\diamle R} \cap \@H_\conn$ with $K$ a successor of $H$ has $d(\partial_\v H, \partial_\v K) \le S$.
\end{lemma}
\begin{proof}
Take $R$ satisfying \cref{thm:kron-moller}.
Let $H \subsetneq K \in \@H^*_{\diamle R} \cap \@H_\conn$ with $d(\partial_\v H, \partial_\v K) > S$; thus $\Ball_S(H) \subseteq K$.
Pick any $x \in \partial_\v H$.
Then $\Ball_R(x)$ contains $\partial_\v H$ and is contained in $\Ball_{S-1}(H)$, so $\Ball_1(\Ball_R(x)) \subseteq K$, i.e., $\Ball_R(x) \cap \Ball_1(\neg K) = \emptyset$.
Thus $\Ball_1(\neg K)$ is contained in a single component $L$ of $\neg \Ball_R(x)$, with $H \cap L = \emptyset$ since $H \cap \neg K = \emptyset$ and $\partial_\v H \subseteq \Ball_R(x)$ whence every component of $\neg \Ball_R(x)$ is contained in or disjoint from $H$.
Then $\neg L \in \@H^*_{\diamle R} \cap \@H_\conn$ by \cref{thm:kron-moller} with $H \subsetneq H \cup \Ball_R(x) \subseteq \neg L \subseteq \neg \Ball_1(\neg K) \subsetneq K$, so $K$ is not a successor of $H$.
\end{proof}

\begin{corollary}
\label{thm:quasitree-bddeg}
If $G$ is a bounded degree quasi-tree, then for all sufficiently large $R < \infty$, we have:
\begin{enumerate}[label=(\alph*)]
\item \label{thm:quasitree-bddeg:dense}
$\@H_{\diamle R}$ is dense towards ends of $G$.
\item \label{thm:quasitree-bddeg:median-bddeg}
$\@U^\circ(\@H_{\diamle R} \cap \@H_\conn)$ is a bounded degree median graph with hyperplanes of bounded size (or equivalently diameter).
\item \label{thm:quasitree-bddeg:qi}
$x |-> \^x : X -> \@U^\circ(\@H_{\diamle R} \cap \@H_\conn)$ is a quasi-isometry.
\end{enumerate}
\end{corollary}
\begin{proof}
The above lemma shows that for $H, K \in \@H^*_{\diamle R} \cap \@H_\conn$, if one is a successor of the other, then their boundaries are at bounded distance apart.
Also, the proof of \cref{thm:cut-nonnested} shows that for such $H, K$ which are non-nested, there is a vertex in $\partial_\v K$ which lies on a geodesic $p_{xy}$ between two points $x, y \in \partial_\v H$, hence $H, K$ have boundaries at distance $\le R$ apart.
Thus, $H, K$ which are adjacent in the pocset $\@H_{\diamle R} \cap \@H_\conn$ (in the sense of \cref{thm:hyp-adj}) have boundaries at bounded distance apart.
Since these boundaries also have diameter $\le R$, and $G$ has bounded degree, the adjacency graph on $\@H^*_{\diamle R} \cap \@H_\conn$ has bounded degree.
From the proof of \cref{thm:med-hyp-locfin}, we see that $\@U^\circ(\@H_{\diamle R} \cap \@H_\conn)$ has bounded degree, and also from \cref{thm:med-hyp-fin} that for each of its half-spaces $H$, the median subgraph $\partial_\iv H$ has boundedly many half-spaces, which implies bounded size of $\partial_\iv H$ since $\partial_\iv H `-> 2^{\@H_\cvx(\partial_\iv H)}$ (\cref{thm:half-sep}).

Finally, to show \cref{thm:quasitree-bddeg:qi}: for $x, y \in X$, we have
\begin{equation*}
d(\^x, \^y) = \abs{\{H \in \@H_{\diamle R} \cap \@H_\conn \mid x \in H \not\ni y\}}
\end{equation*}
by \cref{thm:half-dist}.
We may upper-bound this for $x \mathrel{G} y$ since any such $H$ must have $x \in \partial_\v H$, whence $d(\^x, \^y)$ is upper-bounded for all $x, y$ by a multiple of $d(x, y)$.
And we may lower-bound $d(\^x, \^y)$ by a multiple of $d(x, y)$ using \cref{thm:quasitree-cut-adj}, by letting $\neg H \subseteq \neg \{x\}$ be the component (radial cut) containing $y$, so that $\diam(\partial_\v H)$ is bounded by the degree bound of $G$ (which we may assume to be $\le R$); then letting $H \subseteq K \subseteq \neg \{y\}$ be the component containing $\partial_\v H$; and then using \cref{thm:quasitree-cut-adj} to find a sequence $H = H_0 \subsetneq H_1 \subsetneq \dotsb \subsetneq H_n = K$ in $\@H_{\diamle R} \cap \@H_\conn$ with consecutive boundaries at bounded distance, thus with $n$ at least a multiple of $d(x, y)$.
\end{proof}

\begin{corollary}
\label{thm:cber-quasitreeable-bddeg}
Let $(X,E)$ be a CBER, $G \subseteq E$ be a quasi-treeing with components of bounded degree.
Then there is a componentwise bounded degree Borel forest $(Y,T)$ and a Borel reduction $(X,E) -> (Y,\#E_T)$ which is componentwise a quasi-isometry.
\end{corollary}
\begin{proof}
The Borel reduction to a treeable CBER constructed in \cref{thm:cber-quasitreeable-treeable}, ultimately \cref{thm:cber-wallable-treeable}, is componentwise given by the quasi-isometry $x |-> \^x$ to a median graph of the above result, followed by the identity map to the subforest given by \cref{thm:cber-median}.
\end{proof}

Finally, we point out that even though we only considered quasi-trees which are graphs above, we can easily generalize to proper pseudometric spaces:

\begin{lemma}
Let $(X,d)$ be a proper pseudometric space with a coarse equivalence $f : X -> Y$ to a graph $(Y,T)$.
Then for some $R < \infty$, $(X,d)$ is coarsely equivalent (via the identity map) to the distance $\le R$ graph on $X$.
\end{lemma}
\begin{proof}
Let $g : Y -> X$ be a quasi-inverse of $f$, and let $S < \infty$ be such that $d(y,y') \le 1 \implies d(g(y),g(y')) \le S$.
Let $R := \max(d_X(1_X, g \circ f), S)$, and let $G$ be the distance $\le R$ graph on $X$.
Then $g : (Y,T) -> (X,G)$ is $1$-Lipschitz, whence $g \circ f : (X,d) -> (X,G)$ is bornologous, whence so is $1_X$ since $d_G(1_X, g \circ f) \le 1$.
And $1_X : (X,G) -> (X,d)$ is clearly $R$-Lipschitz.
\end{proof}

\begin{corollary}
If $X$ is a proper pseudometric space coarsely equivalent to a (graph-theoretic) tree, then the distance $\le R$ graph on $X$ is a quasi-tree for some $R$.
\qed
\end{corollary}

\begin{corollary}
\label{thm:cber-quasitreeable-metric}
If a CBER $(X,E)$ admits a Borel classwise proper pseudometric $d : E -> [0,\infty)$ with each class coarsely equivalent to a tree, then $E$ is treeable.
If $d$-balls of each radius are of bounded cardinality, then $E$ is even Borel reducible to a forest via a componentwise quasi-isometry.
\end{corollary}
\begin{proof}
By the above, for each class, there is some $R$ such that the distance $\le R$ graph has $\@H_{\diamle S}$ dense towards ends for some $S$; apply \cref{rmk:walls-proper-borel}.
The last statement follows from \cref{thm:cber-quasitreeable-bddeg}.
\end{proof}

\subsection{Tree decompositions and bounded tree-width}

\begin{definition}
Let $(X,G), (Y,T)$ be two connected graphs.
For a binary relation $F \subseteq X \times Y$, let $\^F \subseteq \^X \times \^Y$ denote the closure of the image of $F$ under the canonical map $X \times Y -> \^X \times \^Y$.

Note that a special case is when $F$ is the graph of a function $f : X -> Y$.
In that case, recall from \cref{def:ends} that $\^f$ is a function $\^X -> \^Y$ iff preimage $f^{-1}$ preserves $\@H_\fin$.

In general, the domain of $\^F$ will be a compact set in $\^X$, thus will be all of $\^X$ iff the domain of $F$ is dense in $\^X$.
For locally finite $G$, this just means that the domain of $F$ is all of $X$.
\end{definition}

\begin{definition}
\label{def:treedec}
$F$ as above is a \defn{tree decomposition} if $T$ is a tree, each image set $F(x) \subseteq Y$ is connected, and $x \mathrel{G} x' \implies F(x) \cap F(x') \ne \emptyset$.
If $F^{-1}(y) \subseteq X$ is finite for each $y \in Y$, we say $F$ has \defn{finite width} and that $G$ has \defn{finite tree-width}; if $\abs{F^{-1}(y)} \le N$ for each $y$, we say $F$ has \defn{width $\le N-1$} and that $G$ has \defn{tree-width $\le N-1$}; if this holds for some $N \in \#N$, we say that $F$ has \defn{bounded width} and $G$ has \defn{bounded tree-width}.
\end{definition}

For background on tree decompositions and tree-width, see \cite{Diegraph}.

\begin{lemma}
\label{thm:treedec-fintrees}
Suppose $G$ is a locally finite graph and $F$ as above is a tree decomposition.
Then there is a tree decomposition $F' \subseteq F$ (in particular, $F'$ has width $\le$ that of $F$) such that $F'(x)$ is finite for each $x \in X$.
\end{lemma}
\begin{proof}
Let $X = \{x_0, x_1, \dotsc\}$ be a bijective enumeration.
Let $F_0 := F$, and inductively for each $n$, let $F_{n+1} \subseteq F_n$ be given by $F_{n+1}(x) := F_n(x)$ for $x \ne x_n$, while $F_{n+1}(x_n) \subseteq F_n(x_n)$ is given as follows: for each $y \in \Ball_{\le1}(x_n)$, pick some $z_y \in F_n(x_n) \cap F_n(y)$; then let $F_{n+1}(x_n) \subseteq Y$ be the convex hull of these $z_y$'s.
Then each $F_n$ is a tree decomposition, whence so is $F' := \bigcapdown_n F_n$, since the definition of tree decomposition only requires checking $F'(x), F'(y)$ for each edge $x \mathrel{G} y$.
\end{proof}

\begin{lemma}
\label{thm:treedec-ends}
Suppose $G$ is a locally finite graph and $F$ as above is a finite width tree decomposition.
Then there is a finitely branching convex subtree $Y' \subseteq Y$ such that $F' := F \cap (X \times Y')$ is still a tree decomposition (with width $\le$ that of $F$).
Moreover, $\^F \subseteq \^X \times \^Y$ is a function on ends of $X$, and already maps each such end to an end of $Y'$.
\end{lemma}
\begin{proof}
For each $y \in Y$, since $F$ is a tree decomposition, we have a partition
\begin{align*}
X \setminus F^{-1}(y) = \bigsqcup_{y' \mathrel{T} y} (F^{-1}(\cone_y{y'}) \setminus F^{-1}(y)),
\end{align*}
which is invariant with respect to the induced subgraph on $X \setminus F^{-1}(y)$, whence only finitely many pieces are nonempty.
Since $F$ has finite width, it follows that each $F^{-1}(\cone_y{y'}) \subseteq X$ has finite boundary.
This implies that $\^F$ is a function on ends: it cannot map an end to a vertex, since each $F^{-1}(y)$ is finite; and it cannot map an end $U \in \^X \setminus X$ to two distinct ends in $\^Y \setminus Y$, since they are separated by some half-space $\cone_y{y'} \subseteq Y$ whence $U$ belongs to either $F^{-1}(\cone_y{y'})$ or its complement.
Now if the $y'$th piece of the above partition is empty, that means $F^{-1}(\cone_y{y'}) \subseteq F^{-1}(y)$, whence we may remove $\cone_y{y'}$ from the tree $Y$ and still maintain that we have a tree decomposition $F \cap (X \times (Y \setminus \cone_y{y'})) = {\proj_{Y \setminus \cone_y{y'}}} \circ F$; moreover, $\^F$ did not map any end of $X$ into $\^{\cone_y{y'}}$ (again since $F^{-1}(y)$ is finite).
Let $Y' \subseteq Y$ be the result of removing all such $\cone_y{y'}$ with $F^{-1}(\cone_y{y'}) \setminus F^{-1}(y) = \emptyset$.
\end{proof}

\begin{proposition}
\label{thm:bddtw-dense}
If $G$ is a locally finite connected graph with tree-width $\le N-1$, then $\@H_{\cardle{N}}$ is dense towards ends of $G$.
\end{proposition}
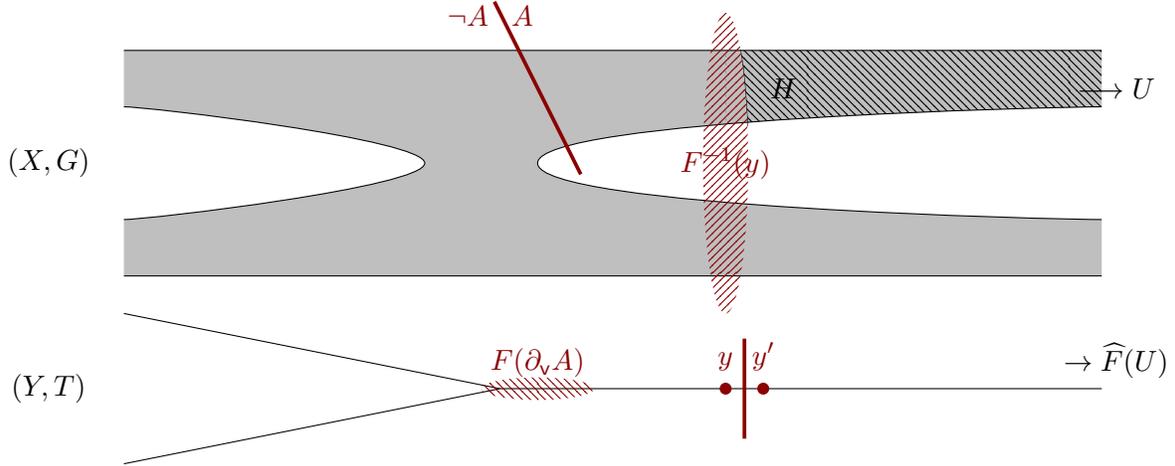
\begin{figure}[htb]
\centering
\begin{tikzpicture}
\node(G) {$(X,G)$};
\coordinate (GO) at ($(G) + (6,0)$);
\coordinate (GL) at ($(GO) + (-5,0)$);
\coordinate (GL1) at ($(GL) + (0,1)$);
\coordinate (GL2) at ($(GL) + (GL) - (GL1)$);
\coordinate (GR) at ($(GO) + (8,0)$);
\coordinate (GR1) at ($(GR) + (0,1)$);
\coordinate (GR2) at ($(GR) + (GR) - (GR1)$);
\def\toparc{($(GR1) + (0,.5)$)
    to ($(GL1) + (0,.5)$)}
\def\leftarc{($(GL1) + (0,-.25)$)
    to[out=0,in=90,looseness=.2] ($(GO) + (-1,0)$)
    to[out=-90,in=0,looseness=.2] ($(GL2) + (0,.25)$)}
\def\botarc{($(GL2) + (0,-.5)$)
    to ($(GR2) + (0,-.5)$)}
\def\rightarc{($(GR2) + (0,.25)$)
    to[out=180,in=-90,looseness=.2] ($(GO) + (.5,0)$)
    to[out=90,in=180,looseness=.2] ($(GR1) + (0,-.25)$)}
\fill[graph interior]
    \toparc -- \leftarc -- \botarc -- \rightarc -- cycle;
\draw \toparc \leftarc \botarc \rightarc;
\node at (GR1) {$\quad \longrightarrow U$};
\draw[graph cut]
    ($(GO) + (0,2)$) -- ($(GO) + (1,0)$);
\node[graph cut, draw=none, below right] at ($(GO) + (0,2.25)$) {$A$};
\node[graph cut, draw=none, below left] at ($(GO) + (0,2.25)$) {$\neg A$};
\node(T) at ($(G) + (0,-4)$) {$(Y,T)$};
\coordinate (TO) at (GO |- T);
\draw (TO)
    edge +($(GL1) - (GO)$)
    edge +($(GL2) - (GO)$)
    edge +($(GR) - (GO)$);
\node[above] at (GR |- T) {$\quad -> \^F(U)$};
\path[pattern=north west lines,pattern color=DarkRed] ($(TO) + (.5,0)$) ellipse[x radius=.75,y radius=.15] node[graph cut, draw=none, above] {$F(\partial_\v A)$};
\begin{scope}[graph cut, draw=none]
\node(y)[thick, dot, label=above:$y$] at ($(TO) + (3,0)$) {};
\node(y')[thick, dot, label=above:$y'$] at ($(y) + (.5,0)$) {};
\draw ($.5*(y)+.5*(y')$) +(0,.5) -- +(0,-.5);
\gdef\yfurc{(GO -| y) ellipse[x radius=.3,y radius=2]}
\path[pattern=north east lines, pattern color=DarkRed] \yfurc node {$F^{-1}(y)$};
\end{scope}
\begin{scope}[even odd rule]
\clip ($(GO -| y) + (0,.5)$) rectangle ($(GR1) + (0,.5)$);
\clip (GO -| y) rectangle ($(GR1) + (0,1)$) \yfurc;
\draw \yfurc;
\pattern[pattern=north west lines] \rightarc -- \toparc -- cycle;
\node[xshift=2em] at (y |- GR1) {$H$};
\end{scope}
\end{tikzpicture}
\caption{The sides $H$ of the parts $F^{-1}(y)$ of a tree decomposition $F$ are dense towards ends.}
\label{fig:bddtw-dense}
\end{figure}
\begin{proof}
By the preceding lemmas, let $F \subseteq X \times Y$ be a tree decomposition with $Y$ locally finite, each $F(x) \subseteq Y$ finite, and width $\le N-1$.
Let $U \in \^X \setminus X$ and $A \in \@H_\fin(X)$ with $U \in \^A$.
Then the end $\^F(U) \in \^Y$ and the finite set $F(\partial_\v A) \subseteq Y$ are separated by some half-space $\cone_y{y'} \subseteq Y$ such that $\^F(U) \in \^{\cone_y{y'}}$ and $F(\partial_\v A) \cap (\{y\} \cup \cone_y{y'}) = \emptyset$.
It follows that $U \in \^{F^{-1}(\cone_y{y'})} \setminus F^{-1}(y)$, which is a union of some of the components of $X \setminus F^{-1}(y)$ (as in the proof of the preceding lemma), while $\partial_\v A \cap (F^{-1}(y) \cup F^{-1}(\cone_y{y'})) = \emptyset$.
Thus the component $H$ of $X \setminus F^{-1}(y)$ containing $U$ is contained in $A$, since it contains $U \in \^A$ and is disjoint from $\partial_\v A$; and $H \in \@H_{\cardle{N}}$ since $\partial_\ov H \subseteq F^{-1}(y)$ which has size $\le N$.
(See \cref{fig:bddtw-dense}.)
\end{proof}

\begin{corollary}
\label{thm:cber-bddtw-treeable}
If a CBER admits a locally finite graphing with components of bounded tree-width, then it is treeable.
\end{corollary}
\begin{proof}
By \cref{thm:bddtw-dense}, \cref{thm:cber-dense-treeable}, and \cref{rmk:walls-proper-borel}.
\end{proof}

\section{Graphs with a distinguished end and hyperfiniteness}
\label{sec:end-hypf}

In this final section, we show that the treeability results in the preceding sections may be adapted to instead show hyperfiniteness, when the graph or wallspace structure we start with is ``one-ended'', or more generally has one distinguished end per component selected in a Borel way.

\subsection{Just one end}
\label{sec:oneended}

The following gives a simple, self-contained reformulation of the machinery in \cref{sec:walls,sec:cuts} in the one-ended case, that does not depend on familiarity with those or any other earlier sections in this paper.
In particular, the words ``end'' and ``wallspace'' here should be treated as merely part of the name for now (their relevance will become clear later).

\begin{definition}
\label{def:walls-oneended}
A \defn{one-ended proper wallspace} is an infinite set $X$ equipped with a set $\@I^* \subseteq 2^X$ of finite nonempty subsets such that
\begin{enumerate}[label=(\roman*)]
\item \label{def:walls-oneended:cofinal}
$\@I$ is cofinal among finite subsets of $X$, i.e., every finite $F \subseteq X$ is contained in some $I \in \@I^*$;
\end{enumerate}
and any of the following equivalent conditions hold:
\begin{enumerate}[resume*]
\refitem{(ii)} \label{def:walls-oneended:finsep}
$\@I^*$ is finitely separating: for any $x, y \in X$, there are only finitely many $J \in \@I^*$ with $x \in J \not\ni y$;
\refitem{(ii$'$)} \label{def:walls-oneended:fincut}
for any $I \in \@I^*$, there are only finitely many $J \in \@I^*$ with both $I \cap J$ and $I \setminus J$ nonempty;
\refitem{(ii$''$)} \label{def:walls-oneended:finnest}
for any $I \in \@I^*$, there are only finitely many $J \in \@I^*$ non-nested with $I$.
\end{enumerate}
\end{definition}

The terminology and the equivalence of these conditions are justified by \cref{thm:walls-oneended} below.

\begin{example}
Let $(X,G)$ be a one-ended locally finite quasi-tree, $\@I^*$ be the collection of all sets which are the union of some $\Ball_S(x)$ and all finite components of its complement, i.e., the complements of the radial cuts as in \cref{thm:kron-moller} which are neighborhoods of the unique end.
Then \cref{def:walls-oneended:cofinal} clearly holds, by taking $S -> \infty$ for fixed $x$.
And \cref{def:walls-oneended:finsep} also holds using \cref{thm:kron-moller}, since any $I \in \@I^*$ separating $x, y$ must have a boundary point in $\Ball_{d(x,y)}(x)$.
(Via \cref{thm:walls-oneended}, this is a special case of \cref{thm:quasitree-dense}.)
\end{example}

\begin{example}
Let $(X,G)$ be a one-ended connected locally finite graph with tree-width $\le N-1$, $\@I^*$ be the set of all finite $I \subseteq X$ with both $I, \neg I$ connected and $\abs{\partial_\iv I} \le N$.
This corresponds via \cref{thm:walls-oneended} to $\@H_{\cardle{N}} \cap \@H_\conn$, which is a proper walling by \cref{thm:bddtw-dense} (and its proof which shows that it is enough to take $\abs{\partial_\iv I}$ rather than $\min(\abs{\partial_\iv I}, \abs{\partial_\ov I})$).
\end{example}

\begin{example}
\label{ex:poset-oneended}
Let $(X,\le)$ be an infinite poset, and let $\@I^*$ be the set of all principal ideals $\down x := \{y \in X \mid y \le x\}$, for all $x \in X$.
The conditions in \cref{def:walls-oneended} translate to:
\begin{enumerate}
\item[(0)]
all principal ideals $\down x \subseteq X$ are finite;
\refitem{(i)}
$X$ is directed, i.e., every finite subset has an upper bound;
\refitem{(ii)}
for any $x, y \in X$, there are only finitely many $x \le z \not\ge y$;
\refitem{(ii$'$)}
for any $x \in X$, there are only finitely many $y \not\ge x$ such that $x, y$ have a lower bound.
\end{enumerate}
Indeed, such posets can be regarded as an alternate formalization of the conditions in \ref{def:walls-oneended}, since given any $\@I^*$ obeying those conditions, the poset $(\@I^*,\subseteq)$ will obey these conditions.
\end{example}

\begin{example}
For any infinite set $X$, given a sequence of finite equivalence relations $F_0 \subseteq F_1 \subseteq \dotsb$ with $\bigcupup_n F_n = X^2$ (i.e., a witness to hyperfiniteness of the countable set $X$), letting $\@I^*$ be the set of equivalence classes of all the $F_n$, we clearly have \ref{def:walls-oneended}\cref{def:walls-oneended:finnest}.
Thus, \cref{def:walls-oneended} can be regarded as a generalization of the definition of hyperfiniteness.
\end{example}

\begin{theorem}
\label{thm:walls-oneended-hyperfinite}
Let $(X, \@I^*)$ be a one-ended proper wallspace.
We may canonically construct a sequence of FERs $F_0 \subseteq F_1 \subseteq \dotsb$ with $\bigcupup_n F_n = X^2$ (i.e., a witness to hyperfiniteness).
\end{theorem}
\begin{proof}
Let $x \mathrel{F_n} y  \coloniff  \forall I \in \@I^*\, (\abs{I} > n \implies (x \in I \iff y \in I))$.
By \ref{def:walls-oneended}\cref{def:walls-oneended:finsep}, $\bigcupup_n F_n = X^2$.
And each $F_n$ is a finite equivalence relation, since for any $x \in X$, by \ref{def:walls-oneended}\cref{def:walls-oneended:cofinal}, there is some $I \in \@I^*$ with $x \in I$ and $\abs{I} > n$, whence $[x]_{F_n} \subseteq I$.
\end{proof}

\begin{remark}
\label{rmk:walls-oneended-hyperfinite-rank}
In the above proof, it is not essential to take $\abs{I} > n$ in the definition of $F_n$.
Another canonical choice is to take $I$ of rank $\ge n$ in the well-founded poset $\@I^*$; in other words, we let $\@I_0 := \@I^*$ and $\@I_{n+1} := \@I_n \setminus \{\text{minimal elements of } \@I_n\}$, and take the $I \in \@I_n$ in defining $F_n$.
In this case, the above proof can be regarded as producing a one-ended tree from $\@I^*$ through a ``leaf-pruning'' procedure.
This is best explained from the median graph perspective; see \cref{rmk:median-oneended-pruning}.
\end{remark}

\begin{definition}
A \defn{one-ended proper walling} of a CBER $(X,E)$ is a Borel set $\@I^*$ of finite nonempty subsets of $E$-classes whose restriction to each $E$-class satisfies \cref{def:walls-oneended}.
\end{definition}

\begin{corollary}
If a CBER admits a one-ended proper walling, then it is hyperfinite.
\qed
\end{corollary}

\begin{corollary}
\label{thm:cber-quasitreeing-bddtw-oneended}
If a CBER admits a one-ended locally finite graphing which is either a quasi-treeing or has bounded tree-width, then it is hyperfinite.
\qed
\end{corollary}

\begin{remark}
Similarly, if a CBER $E$ admits a Borel structuring by posets obeying the conditions in \cref{ex:poset-oneended}, then it is hyperfinite.
Conversely, given any one-ended proper walling $\@I^*$ of $E$, we may produce a new CBER Borel bireducible with $E$ which is instead structured by such posets, namely the restrictions of $\@I^*$ to each $E$-class, as described in \cref{ex:poset-oneended}.
\end{remark}

We now show that \cref{def:walls-oneended} indeed corresponds to what the terminology suggests:

\begin{proposition}
\label{thm:walls-oneended}
Let $X$ be an infinite set.
\begin{enumerate}[label=(\alph*)]
\item
For a set $\@I^*$ of finite nonempty subsets of $X$, the conditions in \cref{def:walls-oneended} are indeed equivalent; and if they hold, then putting $\@I := \@I^* \sqcup \{\emptyset\}$, the collection $\@H := \@I \sqcup \neg(\@I)$ (where $\neg(\@I) := \{\neg I \mid I \in \@I\}$) is a proper walling such that the median graph $\@U^\circ(\@H)$ is one-ended.
\item
Conversely, for a proper walling $\@H \subseteq 2^X$ such that $\@U^\circ(\@H)$ is one-ended, every set in $\@H$ is either finite or cofinite; and the set $\@I^*$ of nonempty finite sets in $\@H$ satisfies \cref{def:walls-oneended}.
\end{enumerate}
\end{proposition}
\begin{proof}
First, we verify that \cref{def:walls-oneended:finsep}, \cref{def:walls-oneended:fincut}, and \cref{def:walls-oneended:finnest} in \cref{def:walls-oneended} are equivalent, given \cref{def:walls-oneended:cofinal}.
From \cref{def:walls-oneended:finsep}, we get that for any $I \in \@I^*$, there are only finitely many $J \in \@I^*$ separating one of the finitely many pairs $x, y \in I$, whence there are only finitely many $J \in \@I^*$ with $I \cap J$ and $I \setminus J \ne \emptyset$, yielding \cref{def:walls-oneended:fincut}.
Clearly \cref{def:walls-oneended:fincut} implies \cref{def:walls-oneended:finnest}.
From \cref{def:walls-oneended:finnest}, we get \cref{def:walls-oneended:fincut} since there are only finitely many $J \subseteq I$ and since any $I, J \in \@I^*$ have $\neg I \cap \neg J \ne \emptyset$ by cofiniteness.
And from \cref{def:walls-oneended:fincut}, we get \cref{def:walls-oneended:finsep} by taking $\{x,y\} \subseteq I \in \@I^*$ using \cref{def:walls-oneended:cofinal}.

Now let $\@I^*$ satisfy \ref{def:walls-oneended}.
Then the corresponding $\@H := \@I \sqcup \neg(\@I)$ is a walling since \cref{def:walls-oneended:finsep} easily implies that $\@H$ is also finitely separating.
To show that $\@H$ is a proper walling (\cref{def:walls-proper}):
\begin{itemize}
\item
\cref{def:walls-oneended:finnest} easily implies that every $H \in \@H$ has only finitely many $K \in \@H$ non-nested with $H$.
\item
For any $x \in X$, by \cref{def:walls-oneended:cofinal}, there is some $x \in I \in \@I^*$, whence there are only finitely many $y$ in the $\@H$-block of $x$, as all such $y$ must be in $I$.
\item
Let $H \in \@H^* = \@I^* \sqcup \neg(\@I^*)$; we must show there are only finitely many successors $H \subsetneq K \in \@H^*$.
\begin{itemize}
\item  If $\neg H \in \@I^*$, then $\neg K \subseteq \neg H$, of which there are only finitely many; so suppose $H \in \@I^*$.
\item  All the successors $K \in \@I^*$ of $H$ are pairwise intersecting and incomparable, whence there are only finitely many such $K$ by \cref{def:walls-oneended:fincut}.
\item  Finally, if $\neg K \in \@I^*$ with $K$ a successor of $H$, then in particular, $H \in \@I^*$ is maximal disjoint from $\neg K$.
By \cref{def:walls-oneended:cofinal}, there is some $H \subsetneq I \in \@I^*$, so $\neg K$ intersects $I$ by maximality of $H$, and also $I \cap K \supseteq H \ne \emptyset$, so there are only finitely many such $K$ by \cref{def:walls-oneended:fincut}.
\end{itemize}
\end{itemize}
Thus $\@U^\circ(\@H)$ is a locally finite median graph with finite hyperplanes.
Let $U := \neg(\@I) \in \@U(\@H) \cong \^{\@U^\circ(\@H)}$ (\cref{thm:med-ends}).
Then $U \subseteq \@H$ is not clopen, or else it would have a minimal element, i.e., $\@I^*$ would have a maximal element, contradicting \cref{def:walls-oneended:cofinal}; thus $U$ is an end of $\@U^\circ(\@H)$.
And it is the only end, since any other end $V \in \@U(\@H) \setminus \@U^\circ(\@H)$ must be a nonprincipal filter in $\@H$ by \cref{thm:med-ends-half-basis}, hence $V \subseteq \neg(\@I) = U$, whence $V = U$ since both are orientations.

Conversely, let $\@H$ be a proper walling such that $\@U^\circ(\@H)$ is one-ended.
Then every $H \in \@H$ must be finite or cofinite, or else $\^H, \neg \^H \subseteq \@U(\@H)$ would each contain an end by compactness.
So the set $U \subseteq \@H$ of cofinite sets is the unique end; and $\@I^* := \@H^* \setminus U$ satisfies \ref{def:walls-oneended}\cref{def:walls-oneended:finsep} since $\@H$ does, and satisfies \ref{def:walls-oneended}\cref{def:walls-oneended:cofinal} since any finitely many vertices in $\@U^\circ(\@H)$, in particular of the form $\^x$ for $x \in I$ where $I \in \@I^*$, may be separated from $U$ again by \cref{thm:med-ends-half-basis}.
\end{proof}

\begin{remark}
\label{rmk:median-oneended-pruning}
The proof of \cref{thm:walls-oneended-hyperfinite}, modified to take $I \in \@I^*$ of increasing well-founded rank as in \cref{rmk:walls-oneended-hyperfinite-rank}, may be understood in terms of the associated median graph $\@U^\circ(\@H)$ (where $\@H := \@I \sqcup \neg(\@I)$ and $\@I := \@I^* \sqcup \{\emptyset\}$ as above) as follows.

Let us first assume, by replacing the original set $X$ with $\@U^\circ(\@H)$, that $X$ is itself a one-ended median graph with finite hyperplanes, and $\@H = \@H_\cvx(X)$ is the set of half-spaces; thus $\@I \subseteq \@H$ is the set of finite half-spaces.
Then to construct the witness to hyperfiniteness $F_0 \subseteq F_1 \subseteq \dotsb$ in \cref{thm:walls-oneended-hyperfinite}, we let $\@I_0 := \@I^*$, then repeatedly delete all vertices in $X$ belonging to a minimal half-space $I \in \@I_n$ (the ``leaves''), letting $\@I_{n+1} \subseteq \@I_n$ be the remaining half-spaces, and connect these deleted vertices in $F_{n+1}$ to their projection onto the remaining (convex) set of vertices, yielding a one-ended tree (see \cref{fig:oneended-pruning}).
\begin{figure}[htb]
\centering
\begin{tikzpicture}
\begin{scope}[rotate=-45]
\foreach \i in {0,...,6} {
    \draw[graph edge]
        (\i,\i) rectangle +(1,1)
        (\i,\i+1) rectangle +(1,1)
        (\i,\i+2) rectangle +(1,1)
        (\i+1,\i) rectangle +(1,1);
}
\node at (8,8) {$\dotsb$};
\foreach \i in {0,...,3} {
    \draw[graph cut cross]
        ({max(\i-2,0)},{\i+.5}) -- ({\i+2},{\i+.5});
    \draw[graph cut cross]
        ({\i+.5},{max(\i-1,0)}) -- ({\i+.5},{\i+3});
    \draw[graph cut]
        ({max(\i-2,0)},{\i+.5}) -- ({\i+2},{\i+.5});
}
\foreach \i in {0,...,3} {
    \draw[graph tree]
        (\i,\i) -- +(1,1)
        ({\i+1},\i) -- +(0,1)
        ({\i+2},\i) -- +(0,1)
        (\i,{\i+1}) -- +(1,0)
        (\i,{\i+2}) -- +(1,0)
        (\i,{\i+3}) -- +(1,0);
}
\end{scope}
\end{tikzpicture}
\caption{Pruning ``leaf half-spaces'' from a one-ended median graph to get a one-ended tree.
The thick highlighted edges form the $4$th equivalence relation $F_4$ in the resulting witness to hyperfiniteness.}
\label{fig:oneended-pruning}
\end{figure}

If $\@H$ is not already the half-spaces of a median graph, then the $F_n$ instead identify elements $x \in X$ whose corresponding vertices $\^x$ in the induced median graph $\@U^\circ(\@H)$ are identified by stage $n$ in the tree as above.
However, in this case there might not be a canonical way of representing the $F_n$-classes as elements of $X$.
\end{remark}

\subsection{One selected end}
\label{sec:end}

Finally, we consider a generalization of the preceding subsection to the case where the graph/wallspace may have more than one end, but one particular end has been selected.
We will show that this is still enough to yield hyperfiniteness, essentially by ``moving towards'' the selected end, in a manner similar to (though not directly generalizing) \cref{rmk:median-oneended-pruning}.
Unlike in the preceding subsection, here we find it necessary to use the full machinery of median graphs throughout.

\begin{definition}
Let $(X,G)$ be a median graph with finite hyperplanes, so that $\^X \cong \@U(\@H_\cvx(X))$ (\cref{thm:med-ends}), and let $U \in \@U(\@H_\cvx(X))$.
(We will usually think of $U$ as an end; however, the following also works when $U$ is a ``vertex'', i.e., a principal orientation.)
For each $x \in X$, let
\begin{align*}
\@A_{U,x} := \max(U \setminus \^x) = \{H \in U \mid x \in \partial_\ov H\}
\end{align*}
be the set of maximal half-spaces containing $U$ but not $x$, or equivalently (by \cref{thm:poc-dual}), half-spaces containing $U$ with $x$ on the outer boundary.
\end{definition}

Note that we do \emph{not} need to assume that $G$ is locally finite here.

\begin{lemma}
\label{thm:end-corner}
For each $x$, $\@A_{U,x}$ is a finite set of pairwise non-nested half-spaces.
Thus,
\begin{align*}
T_U(x) := \proj_{\bigcap \@A_{U,x}}(x)
\end{align*}
(exists and) is the corner opposite to $x$ of a cube cut by precisely the hyperplanes in $\@A_{U,x}$, with
\begin{equation*}
\^{T_U(x)} \setminus \^x = \@A_{U,x}.
\end{equation*}
\end{lemma}
\begin{proof}
For distinct $H, K \in \@A_{U,x}$, the corners $H \cap K$ and $\neg H \cap \neg K$ contain $U, x$ respectively, while $H \cap \neg K$ and $\neg H \cap K$ are nonempty since $H, K$ are distinct maximal elements of $U \setminus \^x$, hence incomparable; thus $H, K$ are non-nested.
Thus by (the proof of) \cref{thm:nonnested}, $x$ is a corner of a cube cut by all the hyperplanes in $\@A_{U,x}$, whose other corners are given by projections of $x$ onto finite intersections of half-spaces in $\@A_{U,x}$.
It follows that $\@A_{U,x}$ is finite, or else any $H \in \@A_{U,x}$ would have an infinite-dimensional cube on its boundary.
Thus $\bigcap \@A_{U,x}$ is a convex neighborhood of $U$, in particular nonempty, and so $T_U(x)$ exists and is defined by the last equation $\^{T_U(x)} \setminus \^x = \@A_{U,x}$ since it is the opposite corner of a cube with precisely the hyperplanes in $\@A_{U,x}$.
\end{proof}

By an \defn{orbit} of the transformation $T_U : X -> X$, we mean a connected component of its graph, i.e., $\bigcup_{m \in \#N} T_U^{-m}(\{T_U^n(x) \mid n \in \#N\})$ for some $x \in X$.
The set $\{T_U^n(x) \mid n \in \#N\}$ is called the \defn{forward orbit} of $x$, while $\bigcup_{m \in \#N} T_U^{-m}(x)$ is the \defn{backward orbit} of $x$.

\begin{corollary}
\label{thm:end-corner-monotone}
$U \setminus \^x \supseteq U \setminus \^{T_U(x)} \supseteq U \setminus \^{T_U^2(x)} \supseteq \dotsb$ (with strict inclusion unless $U$ is a vertex), and $\lim_{n -> \infty} \^{T_U^n(x)} = U$.
Thus, the graph of $T_U$ gives each $T_U$-orbit the structure of a directed tree converging to $U$ (or simply a rooted tree, in case $U$ is a vertex).
\end{corollary}
\begin{proof}
From $\^{T_U(x)} \setminus \^x = \@A_{U,x} \subseteq U$,
we get $U \setminus \^x = (U \setminus \^{T_U(x)}) \sqcup \@A_{U,x}$;
thus $U \setminus \^{T_U(x)}$ is $U \setminus \^x$ with its set of maximal elements $\@A_{U,x}$ removed.
It follows that each $H \in U \setminus \^x$ will be removed in some $U \setminus \^{T_U^n(x)}$ (namely $n := d(x, H)$, by considering a geodesic of length $n$ from $x$ to $H$ so that one edge will be removed at each step), i.e., $T_U^n(x)$ eventually enters every neighborhood of $U$, i.e., $\lim_{n -> \infty} \^{T_U^n(x)} = U$.
\end{proof}

\begin{remark}
If we define an \defn{$\ell^\infty$-neighbor of $x$} to mean any vertex $y \in X$ such that $x, y$ are opposite corners of a cube, or equivalently (by \cref{thm:nonnested}) the hyperplanes separating them are pairwise non-nested, then it is easy to see that $T_U(x)$ is the $\ell^\infty$-neighbor of $x$ ``nearest $U$'', i.e., ``between'' all other $\ell^\infty$-neighbors of $x$ and $U$, in the sense of \cref{thm:poc-dual}.
(Indeed, note that the inequality $U \setminus \^x \supseteq U \setminus \^{T_U(x)}$ means that $\betw{\^x--\^{T_U(x)}--U}$ according to the betweenness relation of \cref{thm:poc-dual}, except that $U$ here is an end, not a vertex.
This can be made precise by regarding $\@U(\@H_\cvx(X))$ as the profinite median algebra completion of $X$; see \cite{Bowmed}.)
See \cref{fig:end-corner}.

\begin{figure}[htb]
\centering
\begin{tikzpicture}[dot/.append style={minimum width=4pt}]
\begin{scope}[rotate=-45]
\foreach \i in {0,...,6} {
    \draw[graph edge]
        (\i,\i) rectangle +(1,1)
        (\i,\i+1) rectangle +(1,1)
        (\i,\i+2) rectangle +(1,1)
        (\i+1,\i) rectangle +(1,1);
}
\node at (8.5,8.5) {$\dotsb \quad \to U$};
\draw[graph cut cross]
    (2.5,1) -- (2.5,5) node[right,fill=white,inner sep=0pt,xshift=4pt,yshift=1pt]{$\scriptstyle \in \@A_{U,x}$};
\draw[graph cut cross]
    (2,4.5) -- (6,4.5);
\draw[graph cut]
    (2.5,1) -- (2.5,5);
\node[below,fill=white,inner sep=0pt,yshift=-3pt] at (3,5) {$\scriptstyle T_U(x)$};
\draw[graph tree]
    (0,2) -- (7,9)
    (0,1) -- (7,8)
    (0,0) -- (7,7)
    (1,0) -- (8,7);
\foreach \i in {0,...,6} {
    \draw[graph tree]
        ({\i+2},\i) -- +(0,1)
        (\i,{\i+3}) -- +(1,0);
}
\node[dot] at (2,4) {};
\node[below] at (2,4) {$\scriptstyle x$};
\node[dot] at (3,5) {};
\end{scope}
\end{tikzpicture}
\caption{The orbits (thick highlighted edges) of $T_U$, for the one-ended median graph in \cref{fig:oneended-pruning}.}
\label{fig:end-corner}
\end{figure}
\end{remark}

\begin{lemma}
\label{thm:end-corner-boundary}
For any $H \in \@H_\cvx(X) \setminus U$ and $x \in H$, there is an $n \in \#N$ such that $T_U^n(x) \in \partial_\ov H$.
Thus in particular (by finite hyperplanes), $H$ intersects only finitely many $T_U$-orbits.
\end{lemma}
\begin{proof}
Since $x \in H$ but $\lim_{n -> \infty} \^{T_U^n(x)} = U \not\in \^H$, there is a least $n \in \#N^+$ such that $T_U^n(x) \not\in H$, whence $T_U^n(x) = T_U(T_U^{n-1}(x)) \in \partial_\ov H$ since it is a corner of a cube opposite to $T_U^{n-1}(x) \in H$.
\end{proof}

\begin{remark}
An alternative to $T_U$, that also satisfies \cref{thm:end-corner-monotone,thm:end-corner-boundary} which are all we need for this argument, is to simply move each $x \in X$ across \emph{one} arbitrarily chosen hyperplane in $\@A_{U,x}$.
This has the disadvantage of being not completely canonical (i.e., automorphism-invariant, despite still working in the Borel context), but the advantage of producing a subforest of the original graph $G$.
\end{remark}

\begin{lemma}
\label{thm:end-corner-root}
For any nonempty $C \subseteq X$ contained in a single $T_U$-orbit and also contained in some $H \in \@H_\cvx(X) \setminus U$, the intersection of the forward $T_U$-orbits of all $x \in C$ is the forward $T_U$-orbit of a single vertex $r_C \in X$, which we call the \defn{$T_U$-root of $C$}.
\end{lemma}
\begin{proof}
For each $x \in C$, let $n_x \in \#N$ such that $T_U^{n_x}(x) \in \partial_\ov H$.
Then $\{T_U^{n_x}(x) \mid x \in C\} \subseteq \partial_\ov H$ is a finite set of vertices.
If it is a single vertex, then the $T_U$-root of $C$ is in its backward $T_U$-orbit.
Otherwise, the first vertex in the forward $T_U$-orbits of all the $T_U^{n_x}(x)$ is the $T_U$-root of $C$.
\end{proof}

\begin{theorem}
\label{thm:end-finoverhypf-hypsm}
Given a median graph $(X,G)$ with finite hyperplanes and a selected end $U$, either
\begin{enumerate}[label=(\roman*)]
\item  $T_U$ has only finitely many orbits (yielding a ``witness to finite-index-over-hyperfiniteness''); or
\item  $T_U$ has infinitely many orbits, in which case we may canonically construct a sequence of equivalence relations $E_0 \subseteq E_1 \subseteq \dotsb$ with $\bigcupup_n E_n = X^2$ and an assignment to each $E_n$-class $C \in X/E_n$ of a finite nonempty set of vertices $R_{n,C} \subseteq X$ (a ``witness to hypersmoothness'').
\end{enumerate}
\end{theorem}
\begin{proof}
Suppose $T_U$ has infinitely many orbits.
Define $E_n \subseteq X^2$ by
\begin{align*}
x \mathrel{E_n} y  \coloniff  \forall H \in \@H_\cvx(X) \setminus U\, (H \text{ intersects more than $n$ $T_U$-orbits} \implies (x \in H \iff y \in H)).
\end{align*}
By \cref{thm:end-corner-boundary,thm:half-dist}, $\bigcupup_n E_n = X^2$.
Since $T_U$ has infinitely many orbits, for any $n \in \#N$, each $x \in X$ belongs to some $H \in \@H_\cvx(X) \setminus U$ intersecting more than $n$ $T_U$-orbits, as we may separate $U$ from $x$ and $n$ other points in distinct orbits by a half-space (by \cref{thm:med-ends-half-basis}).
Thus each $E_n$-class $C \in X/E_n$ is contained in some $H \in \@H_\cvx(X) \setminus U$, and so for each of the finitely many (by \cref{thm:end-corner-boundary}) $T_U$-orbits $O \subseteq X$ intersecting $C$, $O \cap C$ has a $T_U$-root $r_{O \cap C}$ by \cref{thm:end-corner-root}.
Let $R_{n,C}$ be the set of all of these finitely many $T_U$-roots.
\end{proof}

\begin{corollary}
\label{thm:cber-median-end-hypf}
If a CBER $(X,E)$ admits a median graphing $G \subseteq E$ with finite hyperplanes, together with a Borel selection $(U_C)_{C \in X/E}$ of one end in each component, then $E$ is hyperfinite.
\end{corollary}

Here by a ``Borel selection of one end in each component'', we mean that $\bigsqcup_{x \in X} U_{[x]_E} \subseteq \bigsqcup_{x \in X} \@H_\cvx(G|[x]_E) =: \@H_\cvx(G)$ is Borel in the bundle of wallings as in \cref{def:walling}, or equivalently, that $\bigsqcup_{C \in X/E} (U_C \setminus \{\emptyset, C\})$ is Borel in the global space of all nontrivial half-spaces in all components where a half-space is identified with the pair of finite sets given by its inner and outer boundaries.
Equivalently by \cref{thm:med-ends}, if we represent ends $U_C$ as ultrafilters of boundary-finite sets (rather than orientations of half-spaces), this means that $\bigsqcup_{x \in X} U_{[x]_E} \subseteq \bigsqcup_{x \in X} \@H_\fin(G|[x]_E)$ is Borel, or the corresponding condition for the global space of all nontrivial cuts.

\begin{proof}
The transformation $T_U(x) := T_{U_{[x]_E}}(x)$ defined componentwise as in \cref{thm:end-corner} is easily seen to be Borel, and we have an $E$-invariant Borel partition $X = Y \sqcup Z$ into the components on which $T_U$ has finitely many or infinitely many orbits respectively.
On $Y$, the orbit equivalence relation of $T_U$ is hyperfinite by \cite[8.2]{DJKhyp}, and $E$ is finite index over it, hence hyperfinite by \cite[1.3]{JKLcber}.
On $Z$, by running the construction of \cref{thm:end-finoverhypf-hypsm} on each component, we get a sequence of CBERs $E_0 \subseteq E_1 \subseteq \dotsb \nearrow E$ together with $E_n$-invariant Borel maps $x |-> R^n_{[x]_{E_n}}$ taking each $E_n$-class $C$ to a finite nonempty subset of $[C]_E$; such maps are countable-to-1 Borel homomorphisms from $E_n$ to equality (on the space of finite subsets of $X$), hence each $E_n$ is smooth by Lusin--Novikov uniformization (see also \cite[3.37]{Kcber} or \cite[5.8]{CKstr}), whence $E$ is hypersmooth and so hyperfinite by \cite[5.1]{DJKhyp}.
\end{proof}

\begin{corollary}
\label{thm:cber-cuts-end-hypf}
If a CBER $(X,E)$ admits a graphing $G \subseteq E$ with a Borel walling of cuts $C |-> \@H(C)$ dense towards ends in each component $C \in X/E$ (as in \cref{thm:cber-cuts-dense-treeable}), as well as a Borel selection $(U_C)_{C \in X/E}$ of an end in each component, then $E$ is hyperfinite.
\end{corollary}

Here, as above, ``Borel selection of one end'' can be naturally interpreted by treating ends as ultrafilters of boundary-finite sets.
It is also easily seen to mean equivalently that the set of infinite $G$-rays converging to a chosen end is a Borel set in $X^{\#N}$, which is the definition used in e.g., \cite{Mends}.

\begin{proof}
The dual median graph $\@U^\circ(\@H(C))$ for each $C \in X/E$ constructed in \cref{thm:cber-cuts-dense-treeable}, ultimately \cref{thm:cber-wallable-treeable}, has ends in natural bijection with those of $G|C$ by \cref{thm:cut-proper}, namely via pulling back each end (ultrafilter) $U \in \^C \subseteq 2^{\@H_\fin(C)}$ along the inclusion $\@H(C) `-> \@H_\fin(C)$ to get an orientation in $\@U(\@H(C)) \subseteq 2^{\@H(C)}$ (\cref{rmk:graph-walls}); this easily implies that we may transport the Borel end selection $(U_C)_{C \in X/E}$ of $G$ to a Borel end selection of the Borel median graphing constructed in \cref{thm:cber-wallable-treeable}, which is thus hyperfinite by \cref{thm:cber-median-end-hypf}, whence so is the Borel bireducible $E$.
\end{proof}

\begin{corollary}[of \cref{thm:cber-cuts-end-hypf}, \cref{thm:quasitree-dense}, and \cref{thm:bddtw-dense}]
\label{thm:cber-quasitreeing-bddtw-end-hypf}
If a CBER admits a locally finite graphing which is either a quasi-treeing or has bounded tree-width, as well as a Borel selection of an end in each component, then it is hyperfinite.
\qed
\end{corollary}

\def\MR#1{}
\bibliographystyle{amsalpha}
\bibliography{refs}

\medskip
\noindent
Department of Mathematics\\
University of Michigan\\
Ann Arbor, MI, USA\\
\nolinkurl{ruiyuan@umich.edu}

\medskip
\noindent
Department of Mathematics and Statistics\\
McGill University\\
Montreal, QC, Canada\\
\nolinkurl{antoine.poulin@mail.mcgill.ca}\\
\nolinkurl{anush.tserunyan@mcgill.ca}

\medskip
\noindent
Department of Mathematical Sciences\\
Carnegie Mellon University\\
Pittsburgh, PA, USA\\
\nolinkurl{rant2@andrew.cmu.edu}

\end{document}